\documentclass{article}

\usepackage{hyperref} 
\usepackage{amsmath}
\usepackage{amsthm}
\usepackage{amssymb}
\usepackage{cite}
\usepackage{amsmath,amsthm}
\usepackage{amsfonts}
\usepackage{amssymb}
\usepackage{mathtools}
\usepackage{tikz}
\usepackage{enumitem}
\usepackage{lipsum}
\newcommand\blfootnote[1]{%
  \begingroup
  \renewcommand\thefootnote{}\footnote{#1}%
  \addtocounter{footnote}{-1}%
  \endgroup
}

\usepackage{import}
\newtheorem{theorem}{Theorem}[section]

\newtheorem{definition}[theorem]{Definition}

\newtheorem{conjecture}[theorem]{Conjecture}
\newtheorem{corollary}[theorem]{Corollary}
\newtheorem{proposition}[theorem]{Proposition}
\newtheorem{lemma}[theorem]{Lemma}
\newtheorem{problem}[theorem]{Problem}
\newtheorem{claim}[theorem]{Claim}

\usepackage{lineno}

\begin{document}

\title{\bf A Proof of the Alternate Thomass\'e Conjecture for Countable $N$-Free Posets \blfootnote{This paper is a project as part of author's thesis under supervision of Dr. Claude Laflamme and Dr. Robert Woodrow at the Department of Mathematics and Statistics, University of Calgary, Calgary, AB, Canada (2017-2022).}}

\author{Davoud Abdi} 

\maketitle              

\begin{abstract}
An $N$-free poset is a poset whose comparability graph does not embed an induced path with four vertices. 
We use the well-quasi-order property of the class of countable $N$-free posets and some labelled ordered trees to show that a countable $N$-free poset has one or infinitely many siblings, up to isomorphism. This, partially proves a conjecture stated by Thomass\'e for this class.
\end{abstract}



\section{Introduction} 

Two structures $R$ and $S$ are called {\em equimorphic}, denoted by $R\approx S$, when each embeds in the other; we may also say that one is a {\em sibling} of the other. Equimorphic finite structures are necessarily isomorphic, but this is no longer the case for infinite structures. For instance, the rational numbers, considered as a linear order, has continuum many siblings, up to isomorphism. The number of isomorphism classes of siblings of a structure $S$ is called the {\em sibling number} of $S$ and we denote it by $Sib(S)$. 
A {\em relation} $R$ is a pair $(V,E)$ where $V$ is a non-empty set, called the {\em domain} of $R$, and $E\subseteq V^n$ for some positive integer $n$; the integer $n$ is called the arity of $E$ and $E$ is called an $n$-ary relational symbol. A {\em binary relation} is a relation $(V,E)$ where $E$ is of arity 2. The cardinality of a relation is the cardinality of its domain. 
When studying infinite relations, Thomass\'e \cite{TH2} made the following conjecture:

\begin{conjecture} [Thomass\'e's Conjecture, \cite{TH2}]
Let R be a countable relation. Then $Sib(R)=1$, $\aleph_0$ or $2^{\aleph_0}$. 
\end{conjecture}

There is also an alternate version of the conjecture as follows:

\begin{conjecture}[The Alternate Thomass\'e Conjecture]
Let R be a relation of any cardinality. Then $Sib(R)=1$ or $\infty$.  
\end{conjecture} 

The alternate Thomass\'{e} conjecture is connected to a conjecture of Bonato-Tardif \cite{BT} ({\em The Tree Alternative Conjecture}) stating that the sibling number of a tree of any cardinality is one or infinite in the category of trees. The connection is through this observation that each sibling of a tree $T$ is a tree if and only if $T\oplus 1$, the graph obtained by adding an isolated vertex to $T$, does not embed in $T$. Therefore, for a tree $T$ not equimorphic to $T\oplus 1$, the tree alternative conjecture and the alternate Thomass\'e conjecture are equivalent. The tree alternative conjecture has been verified for rayless trees by Bonato and Tardif \cite{BT}, for rooted trees by Tyomkyn \cite{TY} and for scattered trees by Laflamme, Pouzet and Sauer \cite{LPS}. In parallel, the alternate Thomass\'e conjecture has been proved for rayless graphs by Bonato, Bruhn, Diestel and Spr\"ussel \cite{BBDS}, for countable $\aleph_0$-categorical structures by Laflamme, Pouzet, Sauer and Woodrow \cite{LPSW}, for countable cographs by Hahn, Pouzet and Woodrow \cite{HPW} and for countable universal theories by Braunfeld and Laskowski \cite{BL}. Furthermore, Thomass\'e's conjecture for countable chains and the alternate Thomass\'e conjecture for all chains have been proved by Laflamme, Pouzet and Woodrow \cite{LPW}. A first step towards partial orders was taken by Abdi \cite{A} who proved Thomasse\'e's conjecture for countable direct sums of chains and the alternate Thomass\'e conjecture for all direct sums of chains. As a generalisation of direct sums of chains, one of the important classes of partial orders is the class of $N$-free partial orders. This class is also a generalisation of the class of series-parallel partial orders which has motivated many studies in order-theoretic mathematics. In this paper, we show that a countable $N$-free poset has one or infinitely many siblings, up to isomorphism. This shows that the class of countable $N$-free partial orders fits into the sibling programme.

While working on the conjectures above, through a personal conversation with Tyomkyn \cite{TY1}, the author was informed of an unpublished manuscript of Tateno \cite{TAT} claiming a counterexample to the Bonato-Tardif conjecture.
Together with Laflamme, Tateno and Woodrow \cite{ALTW}, we revisited and verified Tateno's claim and constructed locally finite trees having an arbitrary finite number of siblings. Moreover, using an adaptation, the tree examples (in order theoretic sense) can provide countable partial orders with a similar conclusion which disproves Thomass\'e's conjecture (see \cite{ALTW}, section 3). It turns out that both conjectures of Bonato-Tardif and Thomass\'e  are false. 
This is a major development in the program of understanding siblings of a given mathematical structure. While counting the number of siblings provides a good first insight into the siblings of a structure, in particular understanding which structures exactly satisfy the conjectures, the equimorphy programme is now ready to move on and focus on the actual structure of the siblings of those structures.
 
\section{Overview}

In order-theoretic mathematics, the class of series-parallel posets is the smallest class of finite posets obtained from the one element poset by iteration of two operations: direct sum and linear sum. This class of posets was introduced by Lawler and studied by Valdes and Tarjan.  
Grillet and Heuchenne introduced $N$-free posets and later Habib and Jegou \cite{HJ} showed that $N$-free posets are generalisations of series-parallel posets by giving a recursive construction of $N$-free posets.  
Finite and some infinite $N$-free posets can be obtained by direct sum and linear sum. However, there are infinite $N$-free posets which need another operation, namely sum over a labelled chain in which the chain has no least element. In order to count the siblings of countable $N$-free posets, we first determine their structure. For this purpose, we use modular decomposition of $N$-free posets in Section \ref{DT}. Section \ref{StructureNE} determines the structure of $N$-free posets by classifying them. 
Then Section \ref{SibNE} will provide a proof of the main result of this article: a countable $N$-free poset has one or infinitely many siblings. In Section \ref{Finiteposetsub}, we will show that a countable $N$-free poset $P$ whose comparability graph has one sibling is a finite poset substitution of chains or antichains and that Thomass\'{e}'s conjecture holds for $P$. This is a generalisation of a result of \cite{LPW} for countable chains whose comparability graph is a clique. More, the result of Section \ref{Finiteposetsub} leaves Thomass\'e's conjecture open for a countable $N$-free poset whose comparability graph has infinitely many siblings.

\section{Partial Orders and $N$-Free Posets}

A \textit{poset (partially ordered set)} is a binary relation $P=(V,\leq)$ where $\leq$ is reflexive, antisymmetric and transitive. If a poset $P$ is given, by $\leq_P$, we mean the order of $P$. Let $P=(V,\leq)$ be a poset. By an element $x$ of $P$ we mean $x\in V$. By $\leq^*$ we mean the partial order obtained by reversing the order $\leq$, that is $x\leq^* y$ if and only if $y\leq x$ where $x, y\in P$. We denote the poset $(V,\leq^*)$ by $P^*$. Any substructure, resp induced substructure, of $P$ is called a \textit{subposet},  an {\em induced subposet}, of $P$. For two subposets $P_1$ and $P_2$ of $P$ we write $P_1 \leq P_2$ if for every $x\in P_1$ and every $y\in P_2$ we have $x \leq y$. Note that $P_1$ or $P_2$ might be empty. The empty set is both greater than and less than any set. Therefore, $\leq$ on set is not a partial order.  Two elements $x, y \in P$ are called \textit{comparable} if $x \leq y$ or $y \leq x$, otherwise they are \textit{incomparable}, denoted by $x \perp y$. A {\em chain},  an {\em antichain}, is a poset whose elements are pairwise comparable,  incomparable. Two  subposets $P_1, P_2$ of $P$ are {\em comparable} if $P_1 \leq P_2$ or $P_2 \leq P_1$, and {\em incomparable}, denoted by $P_1\perp P_2$, if $x\perp y$ for every $x\in P_1$ and every $y\in P_2$. Two elements $x, y \in P$ are \textit{compatible} if the pair $\{x, y\}$ has a lower bound, that is, there is $z\in P$ such that $z\leq x$ and $z\leq y$.  The graph $CG(P)=(V,E)$ such that for two distinct $x, y \in P$, $xy \in E$ if $x < y$ or $y < x$, is called the \textit{comparability graph} of $P$. Note that a partial order is reflexive meaning that each element of a poset is comparable to itself, however, in order to get simple graphs, our definition of a comparability graph prevents having loops in the graphs. A \textit{connected},  \textit{disconnected}, poset is a poset whose comparability graph is connected,  disconnected. By a \textit{component} of a poset $P=(V,\leq)$, we mean an induced subposet $Q=(W,\leq)$ of $P$ where $W$ is the vertex set of a component of $CG(P)$. We say two elements $x, y$ of $P$ are connected if they belong to the same component of $P$; equivalently, they are connected in $CG(P)$ by a path and we say $x, y$ are disconnected when they belong to different components of $P$. 

Let us denote by an `$N$' the following poset on four elements $\{a, b, c, d\}$ such that $a < b$, $c < b$, $c < d$, $a\perp c$, $b \perp d$ and $a \perp d$. An $N$-\textit{free poset} is a poset which does not embed an $N$. For instance, a chain is an $N$-free poset.

\subsection{Graph Substitution and Poset Substitution} \label{GPSubstitution}

Let $G$ be a graph. By $V(G)$, resp $E(G)$, we mean the set of vertices, resp edges, of $G$. For simplicity, by $x\in G$ we mean that $x$ is a vertex of $G$. The following definition is from \cite{CD}. 

\begin{definition} [Graph Substitution] \label{Graphsub}
Let $K$ be a graph and let $(H_v)_{v\in K}$ be a pairwise disjoint family of graphs. We denote by $G:=K[H_v/v : v\in K]$ the graph resulting from the substitution of $H_v$ for $v\in K$ as follows: 
\begin{enumerate}
    \item $V(G)=\bigcup_{v\in K} V(H_v)$;
    \item $xy \in E(G)$ if and only if either
    \begin{enumerate}
        \item $xy\in E(H_v)$ for some $v\in K$; or
        \item for two distinct $u, v \in K$ we have $x \in H_u$, $y \in H_v$ and $uv\in E(K)$. 
    \end{enumerate}
\end{enumerate} 
Then we say $G$ is a {\em graph substitution} of the $H_v$ for $K$. We call $K$ the {\em context graph} and the $H_v$ the {\em graph blocks} of $G$. If $K$ is a clique (complete graph), resp an independent set, then $G$ is called a {\em complete}, resp {\em direct}, {\em sum} of graphs.
\end{definition} 

A poset substitution applies a {\em composition operation} which is the substitution of a poset $P$ for an element $x$ in a poset $Q$. In the resulting poset, denoted by $Q[P/x]$, the element $x$ is replaced by $P$, and all elements in $Q$ comparable to $x$ are comparable to all elements of $P$. 
A \textit{chain} ({\em linearly ordered set}),  an \textit{antichain}, is a poset whose elements are pairwise comparable,  incomparable. Any substructure of a chain is called a {\em subchain}.

\begin{definition} [Poset Substitution] \label{PosetSub} 
Let $Q$ be a poset and $\{P_v\}_{v\in Q}$ be a family of pairwise disjoint posets. For each $v\in Q$ let $\leq_v$ be the order of $P_v$. The {\em poset substitution} of the $P_v$ for $Q$ is a poset, denoted by $P:=Q[P_v/v : v \in Q]$, defined as follows:  
\begin{enumerate}
    \item $P=\bigcup_{v \in Q} P_v$
    \item for $x, y \in P$, $x \leq_P y$ if and only if 
     \begin{enumerate} 
        \item $x, y \in P_v$ for some $v \in Q$ and $x \leq_v y$; or
        \item for two distinct $u, v\in Q$, we have  $x \in P_u$, $y \in P_v$ and $u \leq_Q v$. 
    \end{enumerate} 
\end{enumerate}
The poset $Q$ is called the {\em context poset} of $P$ and each $P_v$ is called a {\em poset block} of $P$.   If $Q$ is a chain,  an antichain, $P$ is called a {\em linear},  {\em direct, sum} of posets, denoted by $P=+_{v\in Q} P_v$, resp $P=\bigoplus_{v\in Q} P_v$ and each $P_v$ is called a {\em summand}, resp {\em component}, of the linear, resp direct, sum.
\end{definition}

\section{Modular Decomposition} \label{DT}

Courcelle and Delhomm\'{e} \cite{CD} investigated the modular decomposition of infinite (mainly countable) graphs. The notion of modular decomposition is essential for establishing structural properties of graphs and related objects, in particular of partial orders and their comparability graphs. This fact is a special case of a general result about modular decomposition of binary relations. In particular,
each $N$-free poset can be represented by a special type of ordered and labelled tree which is called \textit{decomposition tree}.  Due to the importance of such trees,  in this section, we give a detailed presentation of them. The materials  of this section are taken from \cite{CD} and \cite{HPW}. Indeed, we construct the decomposition tree of an $N$-free poset in a way similar to the method employed by \cite{HPW} for cographs. However, one can find more details about the decomposition tree of a countable graph in \cite{CD}.

\subsection{Modules} \label{Modules}

Let $S$ be a set. A \textit{binary structure over} $S$ is a pair $\mathbb{B}:=(V,d)$ where $V$, the {\em domain} of $\mathbb{B}$, is a non-empty set and $d$ is a map from $V\times V$ into $S$. A subset $M$ of $V$ is a \textit{module} of $\mathbb{B}$ if $d(x,y)=d(x,y')$ and $d(y,x)=d(y',x)$ for every $x \in V\setminus M$ and every $y, y' \in M$. Modules are also called {\em intervals}. The sets $\emptyset$, $V$ and the singletons are called \textit{trivial} modules. A binary structure is called {\em indecomposable} if all its modules are trivial. If moreover it has more than two vertices it is called {\em prime}. A non-empty module is called \textit{strong} if it is comparable w.r.t inclusion to every module which it meets. 



\begin{lemma} [\cite{HPW}] \label{Intersectionunionstrong}
The intersection of any set of strong modules is either empty or a strong module.
\end{lemma}

Let $A$ be a non-empty subset of $V$. Let $S_{\mathbb{B}}(A)$ be the intersection of strong modules of $\mathbb{B}$  containing $A$. We write $S_{\mathbb{B}}(x,y)$ instead of $S_{\mathbb{B}}(\{x,y\})$. According to  Lemma \ref{Intersectionunionstrong}, $S_{\mathbb{B}}(A)$ is a strong module. A module is called \textit{robust} if it is either a singleton or the smallest strong module containing two distinct elements. Hence, if $A$ is a non-trivial robust module, then there exist $x\neq y \in A$ such that $A$ is strong and every strong module containing $x$ and $y$ contains $A$, that is, $A=S_{\mathbb{B}}(x,y)$.  A binary structure $\mathbb{B}$ is called {\em robust} if its domain is a robust module of $\mathbb{B}$.  

Let $A$ be a strong module and $x, y \in A$. We set $x \equiv_A y$ if either $x=y$ or there is a strong module containing $x$ and $y$ and properly contained in $A$. It is proven (Lemma 6.5 in \cite{HPW}) that the relation $\equiv_A$ is an equivalence relation on $A$ whose equivalence classes are strong modules. Also, if $A$ has more than one element, then there are at least two classes if and only if $A$ is robust. Moreover, these classes are the maximal strong modules properly contained in $A$. 
The equivalence classes of a strong module $A$ are called \textit{components} of $A$. We say that a module is {\em limit} if it is a strong module such that none of its non-empty proper strong submodules is maximal (see \cite{HPW}). 

\begin{proposition} [\cite{HPW}] \label{Nonlimitrobust}
Let $\mathbb{B}=(V,d)$ be a binary structure and A a subset of V. Then A is a non-trivial robust module if and only if $A$ is not a limit module.
\end{proposition}


By Proposition \ref{Nonlimitrobust}, a module is limit if and only if it is not robust. 
If $A$ is a robust module of a binary structure with at least two elements, then there are two distinct components $I, J$ of $A$. Let $x, x' \in I$ and $y, y' \in J$. Since $J$ is a module and $x \in A\setminus J$, we have $d(x,y)=d(x,y')$. Since $I$ is a module and $y' \in A\setminus I$, we have $d(x,y')=d(x',y')$. So, $d(x,y)=d(x',y')$ meaning that the values $d(x,y)$ for $x \in I$ and $y \in J$ depend only upon $I$ and $J$. Therefore, the binary structure on $A$ induces a binary structure on the set $A/\equiv_A$ of components of $A$, called the \textit{Gallai quotient} of $A$.  The following proposition is given in \cite{HPW} as an observation without proof, however, we give a proof for completeness.  

\begin{proposition} \label{Gallai}
Let $A$ be a robust module of a binary structure $\mathbb{B}$. The strong modules of the Gallai quotient of A are trivial: the empty set, the whole set $A/\equiv_A$ and the singletons.  
\end{proposition}

\begin{proof}
Let $\mathcal{B}$ be a strong module of $A/\equiv_A$ containing more than one component of $A$. Notice that $A$ has at least two components. Assume that $\mathcal{B}$ contains two distinct components $I$ and $J$ of $A$. Let $K \in A/\equiv_A \setminus \mathcal{B}$ and $I_1, I_2\in \mathcal{B}$. We have $d(K,I_1)=d(K,I_2)$ and $d(I_1,K)=d(I_2,K)$. It means that for every $x \in K$ and every $y, y' \in \bigcup \mathcal{B}$, we have $d(x,y)=d(x,y')$ and $d(y,x)=d(y',x)$. Thus, $\bigcup\mathcal{B}$ is a module of $A$. We prove that $\bigcup\mathcal{B}$ is a strong module of $A$. Let $N$ be a module of $A$ such that $N \cap (\bigcup\mathcal{B}) \neq \emptyset$. Without loss of generality assume that $N$ meets $I$. Since $I$ is strong, $N \subseteq I$ or $I \subseteq N$. In the first case $N \subseteq \bigcup\mathcal{B}$. In the second case if $I \subset N$, then $N=A$ because $I$ is maximal and hence $\bigcup\mathcal{B}\subseteq N$, so, $\bigcup \mathcal{B}$ is strong. Note that the components of $A$ are the maximal strong modules properly contained in $A$. Since $I$ is a proper subset of the strong module $\bigcup\mathcal{B}$, it follows that $\bigcup\mathcal{B}=A$. Equivalently, $\mathcal{B}=A/\equiv_A$.
\end{proof} 

The central result of the decomposition theory of binary structures describes the structure of the Gallai quotient. Let $\mathbb{B}=(V,d)$ be a binary structure. We say that $\mathbb{B}$ is {\em constant} with value $\alpha$ if $d(x,y)=\alpha$ for all $x\neq y\in V$, and {\em linear} with values $\{\alpha, \beta\}$, $\alpha\neq \beta$, if $\{(x,y)\in V\times V : x\neq y, d(x,y)=\alpha\}$ is a linear order and $\{(x,y)\in V\times V : x\neq y, d(x,y)=\beta\}$ is the opposite linear order.

\begin{theorem} [\cite{HPW}] \label{Primeconstantlinear}
The strong modules of a binary structure $\mathbb{B}$ are trivial if and only if $\mathbb{B}$ is prime,  constant or linear. 
\end{theorem} 

Let $A$ be a non-trivial robust module of a binary structure $\mathbb{B}$. The {\em type} of $A$, denoted by $t(A)$, is {\em prime} if $A/\equiv_A$ is prime, otherwise its type is $\alpha$ if $A/\equiv_A$ is  constant with value $\alpha$, and $\{\alpha, \beta\}$ if $A/\equiv_A$ is linear with values $\{\alpha, \beta\}$, $\alpha\neq \beta$. By Proposition \ref{Gallai},  the strong modules of the Gallai quotient $A/\equiv_A$ are trivial. Hence, the Gallai quotient $A/\equiv_A$ is either prime, constant or linear. 

Let $A$ be a non-trivial robust module of an $N$-free poset $P$. Note that $P$ can be regarded as a binary structure $\mathbb{B}=(P,d)$ where for $x, y\in P$, $d(x,y)=0$ if $x\perp y$, $d(x,y)=+1$ if $x <_P y$ and $d(x,y)=-1$ if $y <_P x$. By Proposition \ref{Gallai}, the strong modules of $A/\equiv_A$ are trivial and by Theorem \ref{Primeconstantlinear}, $A/\equiv_A$ is either prime, constant or linear. Further, a theorem due to Kelly \cite{KE} asserts that any indecomposable partial order with at least three elements embeds `$N$'. By definition, a partial order with only two elements is not prime. Therefore, the Gallai quotient $A/\equiv_A$ is either constant or linear. In the first case, the quotient $A/\equiv_A$ is an antichain i.e. $d(I,J)=0$ for each $I\neq J\in A/\equiv_A$. So, the type of $A$ is 0 in this case. In the second case, the quotient $A/\equiv_A$ is a linearly ordered set, that is, for two $I\neq J\in A/\equiv_A$, we have $I <_P J$ or $J <_P I$. Thus, $\{(I,J)\in A/\equiv_A\times A/\equiv_A : I\neq J, d(I,J)=+1\}$ is a linear order and $\{(I,J)\in A/\equiv_A\times A/\equiv_A: I\neq J, d(I,J)=-1\}$ is the opposite linear order. Hence, the type of $A$ is $\{-1,+1\}$ in this case.  

\subsection{Decomposition Tree of $N$-Free Posets} \label{DTNE}

Our terminology in this section borrows from \cite{TH2} and \cite{HPW}. We follow the way of construction employed by \cite{HPW} in which a decomposition tree is a meet-tree ordered by reverse inclusion of robust modules. Let $P=(V,\leq)$ be a poset.  We use the notations $P^x=\{y \in V : x \leq y\}$ and $P_x=\{y \in V : y \leq x\}$. A \textit{forest} is a poset $F$ such that for every $x\in F$, the set $F_x$ is a chain; the elements of $F$ are called \textit{nodes}.  A \textit{tree} is a forest whose elements are pairwise compatible. Let $(T,\leq)$ be a tree. 
A {\em branch} of a node $x\in T$ is an induced and maximal (under inclusion) subtree of $T^x\setminus \{x\}$. A {\em limit branch} is a branch of a node $x\in T$ with no minimal element. For two nodes $x, y$ of $T$, we say $x$ is a \textit{child} of $y$ if $y < x$ and there is no $z\in T$ such that $y < z < x$. A tree is a \textit{meet-tree} if any two nodes $x, y$ have a greatest lower bound, called their \textit{meet} and denoted by $x \wedge y$. 
Let $T$ be a meet-tree. A \textit{leaf} in $T$ is a maximal node, and a \textit{root} is a minimum one. An \textit{internal} node is one that is not a leaf. A tree has at most one root. We observe that if a node $x$ of $T$ is the meet of a finite set $X$ of the maximal nodes of $T$, denoted by $Max(T)$, then $x$ is the meet of a subset $X'$ of $X$ with at most two nodes. 
We say that a meet-tree $T$ is \textit{ramified} if every node of $T$ is the meet of a finite set of maximal nodes of $T$.  

Let $T$ be a ramified meet-tree and $T':=T\setminus Max(T)$. A $\{0,\{-1,+1\}\}$-\textit{valuation} is a map $v :T'\to \{0,\{-1,+1\}\}$. The valuation is \textit{dense} if for every $x<y$ in $T'$, there is some $z$ with $x< z \leq y$ such that $v(z)\neq v(x)$. A ramified meet-tree densely valued by $\{0,\{-1,+1\}\}$  is called a \textit{decomposition tree}. 

Let $P$ be an $N$-free poset. We construct the decomposition tree $T=(R(P), \leq,v)$ of $P$ as follows:
The nodes of the tree $T$ are the robust modules of $P$ denoted by $R(P)$. The order $\leq$ on the nodes of the tree is reverse inclusion meaning that for two robust modules $A, B$ of $P$, $A \leq B$ if and only if $B \subseteq A$. To see that this order gives a tree, first let $A$ be a non-empty robust module of $P$. For two $C_1, C_2 \in R(P)$ with $C_1, C_2 \leq A$ we have $A \subseteq C_1$ and $A \subseteq C_2$. Since $A$ is non-empty, it follows that there exists some $x\in A$. So, $x \in C_1 \cap C_2$. This means that $C_1 \subseteq C_2$ or $C_2 \subseteq C_1$ because both $C_1$ and $C_2$ are strong. Hence, the down-set of a node of the tree $T$ is a chain. Now, suppose that $A, B$ are two robust modules which are not comparable. Let $x \in A\setminus B$ and $y \in B\setminus A$. Then $x$ and $y$ are distinct. Set $C$ to be the robust module determined by $\{x, y\}$. We must have that $C$ is comparable to $A$ and $B$. Moreover, since $x\in A$, $y \in C$ and $y \notin A$, we obtain $A \subset C$. Interchanging the roles of $x$ and $y$, we get $B \subset C$. Under reverse inclusion, $C$ is the meet of $A$ and $B$, that is $C=A \wedge B$. The robust modules of $P$ consisting of singletons are exactly the leaves of $T$. Robust modules of $P$ which are not singletons are valued with one of two values 0 and $\{-1, +1\}$. If $A$ is a non-trivial robust module of $P$, then the {\em value} $v(A)$ of $A$ is the type of $A$ that is $v(A):=\{-1,+1\}$, resp $v(A):=0$, if $t(A)$ is $\{-1,+1\}$, resp 0. It implies that if $A$ is the least strong module containing two distinct elements $x, y$, then $v(A)=\{-1,+1\}$, resp $v(A)=0$, if $x$ and $y$ are comparable, resp incomparable. Note that the map $v$ is defined on the set of nodes of $T$ which are not maximal, that is $v : T\setminus Max(T) \to \{0,\{-1,+1\}\}$.   Given a node $A\in T$ which is a robust module of $P$, the union of the robust modules properly contained in $A$ is a maximal proper strong submodule of $A$ which might be limit itself.

The following proposition implies that the valuation of $T$ is dense. 

\begin{proposition} [\cite{HPW}] \label{Types}
Let A, C be two non-trivial robust modules of a binary structure such that $A\subset C$ and A and C have the same non-prime type $\{\alpha, \beta\}$. Then, there is a robust module B with $A\subset B\subset C$ whose type is distinct from the type of $\{\alpha, \beta\}$.  
\end{proposition}

\begin{corollary} \label{Dense}
Let P be an NE-free poset and T its decomposition tree. The valuation v of T is dense. 
\end{corollary}

\begin{proof}
Recall that for each non-trivial robust module $A$ of $P$, we have either $t(A):=0$, or $t(A):=\{-1,+1\}$. Now, let $A\subset C$ be a non-trivial robust modules of $P$. If $v(A)\neq v(C)$, then by setting $B:=C$, the result follows. Now assume that $v(A)=v(C)$. Then, both $\mathbb{B}_{\upharpoonright A/\equiv_A}$ and $\mathbb{B}_{\upharpoonright C/\equiv_C}$ are either constant or linear. Therefore, $A, C$ have the same non-prime type $\{\alpha,\beta\}$ ($\alpha=\beta=0$ or $\alpha=-1, \beta=+1$). By Proposition \ref{Types}, there exists a robust module $B$ with $A\subset B\subset C$ whose type is distinct from the type $\{\alpha,\beta\}$, that is $t(B)\neq t(A), t(C)$. This implies that $v(B)\neq v(A)$. Therefore, the valuation $v$ is dense.
\end{proof}

\section{Structure of $N$-Free Posets} \label{StructureNE}

In this section we classify $N$-free posets and show that each $N$-free poset is either a direct sum or a linear sum of $N$-free posets or a sum over a labelled chain with no least element. 
Indeed, all the three operations are based on substituting smaller posets for three sorts of posets: chains, antichains and posets obtained from labelled chains. A \textit{cograph} is a graph with no induced subgraph isomorphic to a path $P_4$ with four vertices  \cite{J}.

\begin{proposition} \label{Subcograph}
Let $G=K[H_v/v : v\in K]$ be a graph substitution. $G$ is a cograph if and only if its context graph and graph blocks are cographs. 
\end{proposition}

\begin{proof}
Let $P$ be a copy  of $P_4$. 

($\Rightarrow$) If $G$ is a cograph, then it is clear that all its graph blocks are cographs. If $K$ embeds $P$, then let $v_1, v_2, v_3, v_4$ be the vertices of $P$ in $K$. For every $1\leq i \leq 4$, pick an element $x_i\in H_{v_i}$. Then, the $x_i$ form a copy of $P_4$ in $G$, a contradiction.  

($\Leftarrow$) For the sake of a contradiction, assume that $P$ is embedded in $G$. We show that no graph block $H_v$ contains more than one element of $P$. By assumption, no $H_v$ embeds $P$. Suppose $|P\cap H_v|=3$ and let $a, b, c \in H_v \cap P$. Then $d \in P$ belongs to some $H_u$ where $u \neq v$. Without loss of generality assume that the vertex $d$ is connected to $c$. This means that $uv\in E(K)$. Consequently, $d$ is connected to $a$ and $b$. But, then $a, b, c$ and $d$ do not form a path as $d$ has degree 3. Now suppose $|P\cap H_v|=2$ and let $a,b \in P\cap H_v$. Suppose $b$ and $c$ are adjacent where $c \in H_u$ with $u \neq v$. It follows that $a$ and $c$ are adjacent as well. The vertex $d$ cannot be adjacent to $c$ as then $deg(c)=3$. Therefore, $d$ must be adjacent to  both $a$ and $b$. In this case the vertices $a, b, c$ and $d$ form a cycle. Thus, every vertex of $P$ belongs to a unique $H_v$.  It follows that $K$ itself embeds $P$, a contradiction.   
\end{proof} 

The only partial order $P$ on $P_4$ with $CG(P)=P_4$ results in an `$N$'. On the other hand, the comparability graph of an `$N$' is $P_4$. So, we get the following result. 

\begin{corollary} \label{SubNEposet}
Let $P=Q[P_v/v : v\in Q]$ be a poset substitution. $P$ is $N$-free if and only if its context poset and poset blocks are $N$-free. 
\end{corollary}

$N$-free posets which are not direct or linear sums require poset labelled sums. 
Let $(I,\leq)$ be a chain and $r$ a map from $I$ to $\{-1, 0, +1\}$. Let $Q^I_r=(I,\leq')$ be defined as follows: for $i<j$, $i\perp j$ if $r(i)=0$, $i<'j$ if $r(i)=-1$ and $j<'i$ if $r(i)=+1$. We prove in the following lemma that $Q^I_r$ is an $N$-free poset. Let $\{(P_i,\leq_i)\}_{i\in I}$ be a disjoint family of non-empty posets. We call the poset substitution of the $P_i$ for $Q^I_r$ i.e. $P=Q^I_r[P_i/i:i\in I]$, a {\em poset labelled sum} of the $P_i$ and denote it by $P=\sum_{i\in I} P_i$. The $P_i$ are called the {\em summands} of $P$. We may view such a sum as the poset associated to the {\em  labelled chain} $C:=(I,\leq,\ell)$ where $\ell(i)=(P_i,r(i))$ and denote it by $P=\sum C$. 

\begin{lemma} 
$Q^I_r$ is an NE-free poset. Hence, a poset labelled sum $\sum_{i\in I}P_i$ is NE-free if and only if its summands are NE-free. 
\end{lemma}

\begin{proof}
In order to prove that $Q^I_r=(I,\leq')$ is a poset, it suffices to show that $\leq'$ is transitive. Let $i, j, k\in I$ and assume that $i <' j$ and $j <' k$. If $i<j<k$, then $r(i)=-1$ and if $k<j<i$, then $r(k)=+1$. In both cases we get $i <' k$. Now, assume that $i<j$ and $k<j$. Since $i <' j$, we have $r(i)=-1$. Further, since $(I,\leq)$ is a chain, we have $i<k$ or $k<i$. In the first case, we get $i <' k$. In the second case, we have $k<i<j$ and regarding $j <' k$ we get $r(k)=+1$. Thus, $i <' k$. Note that the case $j<i$, $j<k$ is not possible because then we get $r(j)=+1$ and $r(j)=-1$, a contradiction. 

Now, suppose that a copy of $N$ embeds in $Q^I_r$. Let $i\in N$ be such that $i\leq j$ for each $j\in N$. Such an element exists because $(I,\leq)$ is a chain. Since $i$ is comparable w.r.t $\leq'$ to at least one $j\neq i$ in $N$, we have $r(i)=-1$ or $r(i)=+1$. But then $i$ is comparable w.r.t $\leq'$ to each element of $N$, a contradiction. Thus, $Q^I_r$ is $N$-free. By Corollary \ref{SubNEposet}, it implies that $P=\sum_{i\in I} P_i$ is $N$-free if and only if each $P_i$ is $N$-free.   
\end{proof}

\subsection{Classification} \label{Classification} 

The classification of $N$-free posets is based on the structure of their decomposition tree. Let $I$ be a chain and $r:I\to\{-1,0,+1\}$ a map. We set $I_0:=\{i\in I : r(i)=0\}$ and $I_1:=\{i\in I : |r(i)|= 1\}$. In order to classify $N$-free posets, we will need poset labelled sums indexed by infinite chains $C=(I,\leq,r)$ with no least element where $r$ is a mapping from $I$ to $\{-1, 0 , +1\}$ such that both $I_0$ and $I_1$ are coinitial in $I$. We aim to show that if $T$ has no least element, then $P$ is a poset labelled sum. The crucial ingredient that we will use is the following (see \cite{BD}, page 1747 Lemma 2.2).

\begin{lemma} [\cite{BD}] \label{BD}
Let $\mathbb{B}=(V,d)$ be a binary structure. The collection of robust modules of $\mathbb{B}$ containing a given element is a chain covering the elements of $V$ (i.e. every element of V belongs to some member of the collection). 
\end{lemma}


Let $P$ be an $N$-free poset. If $T$ is the decomposition tree of $P$, then the collection $\mathfrak{C}_x$ of robust modules of $P$ containing a given element $x\in P$ corresponds to a maximal chain in $T$ whose greatest element is $\{x\}$. Assume that $\mathfrak{C}_x$ has no least element and take an arbitrary $y\in P$. Let $\mathfrak{C}_y$ be the collection of robust modules of $P$ containing $y$. Since $\mathfrak{C}_y$ and $\mathfrak{C}_x$ meet at $\{x\}\wedge \{y\}$ in $T$, it follows that $\mathfrak{C}_y$ does not have a least element as well.

\begin{lemma} \label{Posetreduced}
Let P be an NE-free poset such that for some (every) $x\in P$, the chain of robust modules of $P$ containing $x$ has no least element. Then P is the sum $\sum C$ of a labelled chain $C=(I,\leq,\ell)$ where $(I,\leq)$ is an infinite chain with no least element and for each $i\in I$, $\ell(i)=(P_i,r(i))$ where $P_i$ is an NE-free poset and r is a mapping from I to $\{-1,0,+1\}$ such that $I_0=\{i\in I : r(i)=0\}$ and $I_1=\{i\in I : |r(i)|=1\}$ are coinitial in $I$. 
\end{lemma} 

\begin{proof} 
Let $P$ satisfy the condition of the lemma and $\mathfrak{C}$ be the chain of robust modules of $P$ containing $x$ under reverse inclusion. Let $v$ be the valuation of the decomposition tree $T$ of $P$. Let $A\in\mathfrak{C}$ be given. If $v(A)=0$, then let $P_A$ be the poset obtained by restricting $P$ to $A \setminus [x]_{\equiv_A}$ and if $v(A)=\{-1, +1\}$, then let $P_{A^-}$, resp $P_{A^+}$, be the poset obtained by restricting $P$ to $\{y \in A : y\not\equiv_A x\ \text{and}\ y <_P x\}$, resp $\{y\in A : y\not\equiv_A x, \ \text{and}\ x <_P y\}$. For $A\in \mathfrak{C}$ with $v(A)=\{-1, +1\}$, replace $A$ by $A^-$ if $P_{A^+}=\emptyset$, by $A^+$ if $P_{A^-}=\emptyset$ and by the two-element chain $A^-<A^+$ if both $P_{A^-}$ and $P_{A^+}$ are non-empty. Let $I$ consist of $\{x\}$, the $A\in\mathfrak{C}$ with $v(A)=0$ and the replacements $A^-$ and $A^+$ for $A\in\mathfrak{C}$ with $v(A)=\{-1, +1\}$. For $A_1, A_2\in I$, if $A_1$, resp $A_2$, is determined by some robust module $A$, resp $B$, where $A\neq B$, then define $A_1 < A_2$ if and only if $B \subset A$ ($A <_T B$). It follows that $(I,\leq)$ is a chain with no least element because $\mathfrak{C}$ has no least element. 

Let $I$ be the chain constructed above. Define $r:I\to\{-1,0,+1\}$ by $r(A)=0$ if $v(A)=0$, $r(A^-)=-1$ and $r(A^+)=+1$. Since $\mathfrak{C}$ has no least element and $v$ is dense on the elements of $\mathfrak{C}$, it follows that $I_0$ and $I_1$ are coinitial in $I$. 

Let $I$ be the chain constructed above and for each $A\in I$ let $P_A$ be the poset obtained above. Consider $\{P_A\}_{A\in I}$ which is a family of pairwise disjoint $N$-free posets. We have $P=\bigcup_{A\in I} P_A$. Let $A, B\in I$ be given such that $A<B$. If $r(A)=0$ then $v(A)=0$ and we have $A <_T B$. So, $P_A\perp P_B$ in this case. If $r(A)=-1$, resp $r(A)=+1$, then $A$ is determined by a robust module with value $\{-1,+1\}$ and we have $P_A <_P P_B$, resp $P_B <_P P_A$. This means that $P=Q^I_r[P_A/A : A\in I]$, that is $P$ is a poset labelled sum. In other words $P=\sum C$ where $C:=(I,\leq,\ell)$ is a labelled chain with $\ell(A)=(P_A,r(A))$. Note that $(I,\leq)$ is an infinite chain with no least element and $r$ is a mapping from $I$ to $\{-1,0,+1\}$ such that $I_0$ and $I_1$ are coinitial in $I$. 
\end{proof}

\begin{definition} \label{CCGC}
We say that a poset $P$ has {\em property CCGC} if the complement of the comparability graph of $P$ is connected, that is, $(CG(P))^c$ is connected. 
\end{definition}

In Subsection \ref{Linearsum}, when determining the sibling number of a linear sum $P$ of $N$-free posets, we will see that the property CCGC of the summands of $P$ plays a crucial role. We also prove that an $N$-free poset satisfying Lemma \ref{Posetreduced} has property CCGC. 

\begin{proposition} \label{Plsccgc}
Let P be an NE-free poset whose decomposition tree has no least element. Then P is connected and it has property CCGC. 
\end{proposition}

\begin{proof}
Let $T$ be the decomposition tree of $P$. By assumption $T$ has no least element. Set $G:=CG(P)$. Take $x, y\in G$ and consider the robust module $A$ determined by $x$ and $y$. If $v(A)=\{-1,+1\}$, then $x$ and $y$ are comparable meaning that $xy\in E(G)$. If $v(A)=0$, then since $T$ has no least element and $v$ is dense, we can find a robust module $B\supset A$, equivalently $B <_T A$, with $v(B)=\{-1,+1\}$. Take some $z\in B$. Then $x$ and $z$ are comparable as well as $y$ and $z$. Therefore, there is a path in $G$ connecting $x$ and $y$. Hence, $P$ is connected.

Now we prove that $G^c$ is connected. Let $x, y\in G$ be given. Let $A=S_\mathbb{B}(x,y)$ as computed for $G$. If $v(A)=0$, then $x$ and $y$ are incomparable w.r.t $\leq_P$. Thus, $xy\in E(G^c)$ meaning that in this case $x$ and $y$ are connected in $G^c$. If $v(A)=\{-1,+1\}$, then by density of $v$ and the fact that $T$ has no least element we can find a robust module $B$ of $P$ with $B<_T A$ and $v(B)=0$. Take some $z\in B$. We have $z\perp x$ and $z\perp y$. Therefore, $xz, zy\in E(G^c)$ meaning that $x$ and $y$ are connected in $G^c$. This completes the proof. 
\end{proof}

We have the following classification of $N$-free posets which will be used in determining the sibling number of an $N$-free poset.  

\begin{theorem} \label{Pstructure}
Let P be an NE-free poset with more than one element. Then either 
\begin{enumerate}
    \item P is a direct sum of at least two non-empty connected NE-free posets, or
    \item P is a linear sum of at least two non-empty NE-free posets with property CCGC, or
    \item P is the sum $\sum C$ of a labelled chain $C=(I,\leq,\ell)$ such that $(I,\leq)$ is an infinite chain with no first element, and each label $\ell(i)$ is the pair $(P_i,r(i))$ made of a non-empty NE-free poset $P_i$ and an element $r(i)\in\{-1,0 , +1\}$ in such a way that r is a mapping on the chain $(I,\leq)$ such that $I_0=\{i\in I : r(i)=0\}$ and $I_1=\{i\in I : |r(i)|=1\}$ are coinitial in I. 
\end{enumerate}
\end{theorem}

\begin{proof}
Let $T$ be the decomposition tree of $P$. If $P$ is the least element of $T$, then $P$ is a robust module with value  0 or $\{-1, +1\}$. In the first, resp second, case, each maximal proper robust submodule of $P$ is connected, resp has property CCGC, the quotient $P/\equiv_P$ is an antichain, resp a chain, and $P$ is a direct, resp linear, sum of its components, resp summands. Finally, if $T$ has no least element, then by Lemma \ref{Posetreduced}, $P$ satisfies (3). 
\end{proof}

\section{Siblings of Countable $N$-Free Posets} \label{SibNE}

In this section our aim is to prove the following theorem which shows that the alternate Thomass\'{e} conjecture holds for countable $N$-free posets.

\begin{theorem} \label{AltThomasseNEP}
A countable $N$-free poset has one or infinitely many siblings.
\end{theorem}

A {\em quasi-order} is a reflexive and transitive binary relation. A {\em well-quasi-order (wqo)} is a quasi-order $\mathcal{Q}$ such that any infinite sequence of elements of $\mathcal{Q}$ contains an infinite increasing subsequence w.r.t $\leq_\mathcal{Q}$. 
We denote the class of countable $N$-free posets by $\mathcal{N}_{\leq\omega}$. 
Thomass\'{e} \cite{TH2} proved that  $\mathcal{N}_{\leq\omega}$  is well-quasi-ordered under embeddability. Hence, embeddability is a well-founded relation on $\mathcal{N}_{\leq\omega}$. Therefore, in order to prove that a countable $N$-free poset $P$ has one or infinitely many siblings, we suppose that this property holds for all $N$-free posets which strictly embed in $P$ and then show that the property holds for $P$. In order to tackle the problem we use the classification of $N$-free posets. Recall by Theorem \ref{Pstructure} that an $N$-free poset with more that one element is either a direct sum of at least two non-empty connected $N$-free posets or a linear sum of at least two non-empty $N$-free posets with property CCGC, or it is the sum of a labelled chain such that the chain has no first element. We argue by cases.  We start with the last case and prove that such a countable poset has continuum many siblings. Next, direct and linear sums will be considered.

\subsection{Sum Over a Labelled Chain with no Least Element}

Let $P$ be a countable $N$-free poset such that its decomposition tree has no least element. In this case by Lemma \ref{Posetreduced}, $P$ is the sum $\sum C$ of a  labelled chain $C=(I,\ell)$ such that $I$ has no first element and $\ell(i)=(P_i,r(i))$ where $r:I\to\{-1,0,+1\}$ is a map such that $I_0$ and $I_1$ are coinitial in $I$. We show that $P$ has continuum many siblings in this case. The proof is based on constructing continuum many pairwise non-isomorphic labelled chains $C_\alpha$. We quote some results in \cite{HPW} and also translate some results there into the context of $N$-free posets to show that the sums $\sum C_\alpha$ are pairwise non-isomorphic siblings of $\sum C$. We also provide the notions needed. 

A {\em labelled chain} is a pair $C=(I,\ell)$ where $I$ is a chain and $\ell$ is a map from $I$ to a quasi-order. Let $\mathcal{Q}$ be a quasi-order. A {\em $\mathcal{Q}$-embedding}, (or an {\em embedding} when the quasi-order $\mathcal{Q}$ is definite), of a labelled chain  $C=(I,\ell)$ into another labelled chain $C'=(I',\ell')$, denoted by $C\leq C'$, is an order embedding $f:I\to I'$ such that $\ell(i)\leq_\mathcal{Q} \ell'(f(i))$ for all $i\in I$. Two labelled chains $C, C'$ are {\em equimorphic}, denoted by $C\approx C'$, when they are mutually embeddable. They are {\em isomorphic} when there is an isomorphism $\phi$ from $I$ to $I'$ such that $\ell'\circ \phi=\ell$. 
A labelled chain  $C=(I,\ell)$ is {\em additively indecomposable}, or briefly {\em indecomposable}, if for every partition of $I$ into an initial segment $J$ and a final segment $F$, $C\leq C_{\upharpoonright J}$ or $C\leq C_{\upharpoonright F}$. We say that $C$ is {\em left-indecomposable}, resp {\em strictly left-indecomposable}, if $C$ embeds in every non-empty initial segment, resp if $C$ is left-indecomposable and embeds in no proper final segment. The {\em right-indecomposability} and the {\em strict right-indecomposability} are defined similarly. 

The {\em sum} $\sum_{i\in I} C_i$ of labelled chains over a chain (not over a labelled chain) is defined as in the case of chains. If $I$ is the $n$-element chain $\underline{n}:=\{0,1,\ldots, n-1\}$ with $0<1<\cdots<n-1$, then the sum is rather denoted by $C_0+C_1+\cdots+C_{n-1}$. 
Let $C=(I,\ell)$ be a labelled chain and $J, F$ be a partition of $I$ into an initial and a final segment, respectively. Let $C_0, C_1$ be obtained by restricting $C$ to $J, F$, respectively. With the notion of sum, $C$ is indecomposable if  $C=C_0+C_1$ implies that $C\leq C_0$ or $C\leq C_1$. 

Let $C=(I,\ell)$ be a labelled chain. Recall that $I^*$ is obtained from $I$ by reversing the order of $I$. Define $C^*$ to be the labelled chain $(I^*,\ell)$. 
For two labelled chains $C_0$ and $C_1$, we define $C_0+^* C_1$ to be the sum $C_1^*+C_0^*$. For a chain $J$ define $\sum^*_{j\in J} C_j=\left(\sum_{j\in J} C_j\right)^*=\sum_{j\in J^*} C_j^*$. 
A sequence $(C_n)_{n<\omega}$ of labelled chains is {\em quasi-monotonic} if $\{m : C_n\leq C_m\}$ is infinite for each $n<\omega$, its sum $\sum_{n<\omega} C_n$  is right indecomposable or equivalently $\sum_{n<\omega}^* C_n$ is left indecomposable (\cite{HPW}).  

Let $C=(I,\ell)$ be a labelled chain. Two elements $x, y\in I$, $x\leq y$, are {\em equivalent}, denoted by $x\equiv_C y$, if $C$ does not embed in the restriction of $C$ to the interval determined by $x$ and $y$, that is, $C \not\leq C_{\upharpoonright[x,y]}$. If $C$ is indecomposable, then this relation is an equivalence relation. Equivalence classes are intervals of $I$. Further, either all the elements of $I$ are equivalent, that is, there is only one equivalence class, or the quotient $I/\equiv_C$ is dense and if $J$ is any equivalence class, $C$ does not embed into $C_{\upharpoonright J}$. 

According to Laver \cite{LAV1}, the class $\mathcal{C}_{\leq\omega}$ of countable chains quasi-ordered by embeddability is wqo. Moreover, the class $\mathcal{Q}^{\mathcal{C}_{\leq\omega}}$ of countable chains labelled by a wqo $\mathcal{Q}$ is wqo. Hence, the class of countable chains labelled by $\mathcal{N}_{\leq\omega}$ is wqo.

We list the properties we need. These properties are due to Laver \cite{LAV1} and they are given in \cite{HPW}. We quote the proof given in \cite{HPW}. 

\begin{lemma} [\cite{HPW}]  \label{Labelledchainsproperties}
Let $C=(I,\ell)$ be a countable chain labelled over a wqo $\mathcal{Q}$.   
\begin{enumerate}
   \item If I has no least element, then there is some initial segment J of I such that $C_{\upharpoonright J}$ is left indecomposable;
    \item If C is left indecomposable, then C is an $\omega^*$ sum $\sum^*_{n<\omega} C_n$ where each $C_n$ is indecomposable and the set of m such that $C_n$ embeds into $C_m$ is infinite. 
    \end{enumerate}
\end{lemma}

We translate a result in \cite{HPW} to the context of $N$-free posets. 

\begin{lemma}  \label{Finalsegmentsiso} 
Let $P, P'$ be two isomorphic NE-free posets such that their decomposition tree has no least element. If $P=\sum C$ and $P=\sum C'$ where $C=(I,\ell)$ and $C'=(I',\ell')$ are two labelled chains, then there are two infinite initial segments J of I and $J'$ of $I'$ and an isomorphism h of the induced poset labelled chains $C_{\upharpoonright J}$ and $C'_{\upharpoonright J'}$. 
\end{lemma}

\begin{proof}
Since $P$ and $P'$ are isomorphic, their decomposition trees are isomorphic. We may assume that they are identical. Let $T$ be such a tree. Then, $(I,\leq)$ and $(I'\leq')$ correspond to maximal chains in $T$ and by Lemma \ref{Posetreduced} they do not have a first element. If these chains are identical, then there is nothing to prove ($J=J'=I$). If not, then since $T$ has no least element, these chains meet in some element $x\in I\cap I'$ corresponding to a node in $T$. Set $J=J'=I_x$ and let $h$ be the identity on the initial segment $I_x$. 
\end{proof}

The following lemma appears as Lemma 4 in page 39 of \cite{LPSZ}. 

\begin{lemma} [\cite{LPSZ}] \label{Continuummaps}
There is a set of $2^{\aleph_0}$ maps  $f$ in $\{0,1\}^\mathbb{N}$ such that for every pair f, $f'$ of distinct maps, and every isomorphism h from a final segment F of $\mathbb{N}$ onto another $F'$, there is some $n\in F$ such that $f(n)\neq f'(h(n))$. 
\end{lemma} 


The following lemma gives a construction of siblings of a countable $N$-free poset which are  sums of labelled chains with no least element. The proof is adapted with some passages quoted from \cite{HPW}. 

\begin{lemma} \label{Nonisoplc}
Let $C=(I,\ell)$ be a countable labelled chain with no least element where $\ell(i)$ is the pair $(P_i,r(i))$ made of a non-empty NE-free poset and $r(i)\in \{-1,0,+1\}$ such that $I_0=\{i\in I : r(i)=0\}$ and $I_1=\{i\in I : |r(i)|=1\}$ are coinitial in I. Then, there are $2^{\aleph_0}$  labelled chains $C_f$ where $f:\mathbb{N}\to\{0,1\}$, such that $\sum C_f\ncong \sum C_g$ for $f\neq g$. 
\end{lemma}

\begin{proof}
Let $C=(I,\ell)$ be a countable labelled chain satisfying the conditions of the lemma. We select an initial segment $J$ of $I$ such that the restriction $C_J:=(J,\ell_{\upharpoonright J})$ is left indecomposable (this is due to Lemma \ref{Labelledchainsproperties} (1)). Let $Ev(J)$ be the set of $i\in J$ such that $P_i$ is a chain or an antichain of even size. Since $I$ has no least element and $I_0$ and $I_1$ are coinitial in $I$, there is an infinite descending sequence $(a_n)_{n<\omega}$ of elements of $J$ coinitial in $I$ such that $|r(a_n)|=1$. We insert infinitely many elements $b_n, c_n$ such that $b_n$ covers $a_n$ and $c_n$ covers $b_n$. Set $\Bar{J}:=J\cup \{b_n, c_n : n<\omega\}$, $F:=I\setminus J$ and $\Bar{I}:=\Bar{J}\cup F$. To each map $f:\mathbb{N}\to \{0,1\}$ we associate a labelled chain $C_f=(\Bar{I},\ell_f)$ with  $\ell_f(i):=(\Bar{P_i},r_f(i))$ as follows: If $i\in I\setminus Ev(J)$, then we set $\ell_f(i):=\ell(i)$. If $i\in Ev(J)$, then set $r_f(i):=r(i)$ and let $\Bar{P_i}$ be an extension of $P_i$ to an extra element in such a way that the new poset is a chain if $P_i$ is a chain and an antichain, otherwise. If $i\in \{b_n, c_n\}$, recalling that $|r(a_n)|=1$, we set $r_f(b_n)=0$ and $r_f(c_n)=r(a_n)$, $\Bar{P_i}$ being an antichain, resp a chain, of size $2f(n)+2$ if $i=b_n$, resp $i=c_n$. It is clear that $\bar{I}_0=\{i\in \bar{I} : r_f(i)=0\}$ and $\bar{I}_1=\{i\in \bar{I} : |r_f(i)|=1\}$ are coinitial in $\bar{I}$. 



If $f$ and $f'$ are two maps from $\mathbb{N}$ to $\{0,1\}$ such that the sums $\sum C_f$ of $C_f=(\Bar{I},\ell_f)$ and $\sum C_{f'}$ of $C_{f'}=(\Bar{I},\ell_{f'})$ are isomorphic, then according to Lemma \ref{Finalsegmentsiso}, there are two initial segments $\Bar{J}$ and $\Bar{J'}$ of $\Bar{I}$ and an isomorphism $h$ from $\Bar{J}$ onto $\Bar{J'}$ preserving the labels, that is, the labels $\ell_f(i)=(\Bar{P_i},r_f(i))$ and $\ell_{f'}(h(i))=(\Bar{P}_{{h(i)}},r_{f'}(h(i)))$, $i\in \Bar{J}$, are isomorphic, meaning that $\Bar{P_i}$ and $\Bar{P}_{h(i)}$ are isomorphic and $r_f(i)=r_{f'}(h(i))$ for each $i\in \Bar{J}$. If $i=b_n$, resp $i=c_n$, then necessarily $h(b_n)=b_m$, resp $h(c_n)=c_m$, for some $m$. Indeed, $\Bar{P}_{h(i)}$  must be an antichain, resp chain, of even size. So, $h(b_n), h(c_n)\in Ev(\Bar{J})$. Further, recall that for $i\in Ev(J)$, $\Bar{P_i}$ is a chain or antichain of odd size. Since $r_{f'}(h(b_n))=r_f(b_n)=0$, resp $r_{f'}(h(c_n))=r_f(c_n)=\pm 1$ where $\pm 1$ means $-1$ or $+1$, hence $h(b_n)$, resp $h(c_n)$, must be some $b_m$, resp some $c_m$. Consequently, there are two final segments $F$ and $F'$ of $\mathbb{N}$, and an isomorphism $h$ from $F$ to $F'$ such that $f'(h(n))=f(n)$ for every $n\in F$. By applying Lemma \ref{Continuummaps}, we get $2^{\aleph_0}$ pairwise non-isomorphic labelled chains $C_f$ having the same properties as does $C$.
\end{proof} 

In Lemma \ref{Nonisoplc}, we observe that the order on $P_f:=\sum C_f$ extends both $\sum C_{f\upharpoonright\Bar{J}}$ and $\sum C_{f\upharpoonright F}$ where the relation between an element $x$ of $\sum C_{f\upharpoonright \Bar{J}}$ and an element $y$ of $\sum C_{f\upharpoonright F}$ is determined as follows: $x <_{P_f} y$ if and only if $r_f(i)=-1$ where $i\in\Bar{J}$ and $x\in \Bar{P_i}$; and $y <_{P_f} x$ if and only if $r_f(i)=+1$ where $i\in \Bar{J}$ and $x\in \Bar{P_i}$. We denote this by $\mp$ and obtain $\sum C_f=\sum C_{f\upharpoonright\Bar{J}}\mp\sum C_{f\upharpoonright F}$. 

Let $C=(I,\ell)$ be a labelled chain and $n$ be a positive integer. The {\em ordinal product} $\underline{n}.C$ is the labelled chain $(\underline{n}.I,\ell_n)$ where $\underline{n}.I$ is the ordinal product of the $n$-element chain $\underline{n}:=\{0, \ldots, n-1\}$, $0<1<\cdots<n-1$, and the chain $I$, that is, the ordinal sum of $C$ copies of $\underline{n}$, and $\ell_n(m,i)=\ell(i)$ for every $m\in\underline{n}$, $i\in I$.

\begin{lemma} [\cite{HPW}]  \label{ncopies}
Let $C=(I,\ell)$ be a countably infinite labelled chain. If the labels belong to a b.q.o and C is indecomposable, then for every positive integer n, the ordinal product $\underline{n}.C$ embeds in C. 
\end{lemma}

Now we show that the sum of labelled chains $C_f$ in Lemma \ref{Nonisoplc} are equimorphic to the sum of the labelled chain $C$.  

\begin{lemma} \label{ReducedequimorphP}
Let $C=(I,\ell)$ be a countable labelled chain such that I has no first element and the labels belong to $\mathcal{N}_{\leq\omega}\times\{-1, 0,+1\}$. Then, there is an initial segment J of I such that the restriction $C_J=(J,\ell_{\upharpoonright J})$ is left indecomposable. If J is such an initial segment, then for every map $f:\mathbb{N}\to \{0,1\}$,  $\sum C_f\approx \sum C$ where $C_f$ is constructed as in Lemma \ref{Nonisoplc}.
\end{lemma}

\begin{proof}
The existence of $J$ follows from the fact that the class of countable chains labelled by $\mathcal{N}_{\leq\omega}\times \{-1, 0, +1\}$ is b.q.o and Lemma \ref{Labelledchainsproperties} (1). Let $f:\mathbb{N}\to\{0,1\}$ and $C_f=(\Bar{I},\ell_f)$ be the labelled  chain defined in Lemma \ref{Nonisoplc} and set $P_f:=\sum C_f$. We prove that $\sum C_f\leq \sum C$.  Let $C_{f\upharpoonright\Bar{J}}$, resp $C_{f\upharpoonright F}$, be the restriction of $C_f$ to $\Bar{J}$, resp $F$. We have $C_f=C_{f\upharpoonright\Bar{J}}+C_{f\upharpoonright F}$. By the observation given after Lemma \ref{Nonisoplc}, we get $\sum C_f=\sum C_{f\upharpoonright\Bar{J}}\mp\sum C_{f\upharpoonright F}$. 

To conclude, it suffices to prove that $\sum C_{f\upharpoonright\Bar{J}}$ embeds into $\sum C_{\upharpoonright J}$ and $\sum C_{f\upharpoonright F}$ embeds into $\sum C_{\upharpoonright F}$. Only the first statement needs a proof because $f$ does not affect $F$. In order to prove it, we define an auxiliary labelled chain $D=(M,d)$ as follows: the domain $M$ is $\underline{2}.J\cup X$ where $X=\{b_n, c_n : n<\omega\}$, $b_n$ covers $(1,a_n)$ and $c_n$ covers $b_n$. The labelling $d$ is defined by $d(i)=\ell(j)$ if $i=(m,j)\in \underline{2}.J$ and $d(i)=\ell_f(i)$ if $i\in X$. We prove that the following inequalities hold.
$$\sum C_{f\upharpoonright\Bar{J}}\leq\sum D\leq \sum\underline{2}.C_{\upharpoonright J}\leq \sum C_{\upharpoonright J}.$$

For the first inequality, define $h:\Bar{J}\to \underline{2}.J\cup X$ by $h(j)=(0,j)$ for $j\in J$ and $h(i)=i$ for $i\in X$. Then $h$ is an embedding and due to the definition of $d$, it follows that $C_{f\upharpoonright \Bar{J}}\leq D$ and consequently $\sum C_{f\upharpoonright \Bar{J}}\leq \sum D$.

For the second inequality, recall that $C_{\upharpoonright J}$ is left indecomposable. It can be verified that $\underline{2}.C_{\upharpoonright J}$ is left indecomposable.  Thus, by Lemma \ref{Labelledchainsproperties} (2), we can write $\underline{2}.C_{\upharpoonright J}$ as $\sum^*_{n<\omega} C_n$ where each $C_n$ is indecomposable and for each $n<\omega$, the set of $m$ such that $C_n$ embeds into $C_m$ is infinite. For each $n<\omega$, let $J_n$ be the domain of $C_n$ and let $L_n:=J_n\cup \{b_m, c_m : a_m\in J_n\}$. Set $D_n:=(L_n, \ell_{f\upharpoonright L_n})$. Then, $D=\sum^*_{n<\omega} D_n$. We define an embedding from $\sum D$ to $\sum \underline{2}.C_{\upharpoonright J}$ by induction. Let $n<\omega$. First we show that the poset labelled sum $\sum D_n$ embeds in the poset labelled sum obtained by a finite sum of $C_m$, that is $\sum D_n\leq \sum(C_{m+k}\cdots+C_m)$ for some $m, k<\omega$. The  labelled chain $D_n$ consists of $C_n$ plus finitely many elements, each labelled by a 2 or 4-element chain or antichain. Thus, $D_n$ can be written 
\begin{multline*} 
D_{n,0}+\alpha_{(0,0)}+\alpha_{(1,0)}+D_{n,1}+\alpha_{(0,1)}+\alpha_{(1,1)}+\cdots+D_{n,k_n-1}+ \\ \alpha_{(0,k_n-1)}+  \alpha_{(1,k_n-1)}+D_{n,k_n}
\end{multline*}

with $D_{n,0}+\cdots+D_{n,k_n}=C_n$ and for each $0\leq i\leq k_n-1$, $\alpha_{(0,i)}$, resp $\alpha_{(1,i)}$, is a singleton belonging to $X:=\{b_n, c_n : n\in\mathbb{N}\}$ labelled by a 2 or 4-element antichain, resp chain.  Suppose that the poset labelled sum
$$\sum(D_{n,t} + \cdots+D_{n,k_n-1}+\alpha_{(0,k_n-1)}+\alpha_{(1,k_n-1)}+D_{n,k_n})$$
where $0< t \leq k_n$, has been embedded in $\sum(C_{\varphi(t)}+\cdots+C_m)$ for some $m$. 
We know that $\alpha_{(1,t-1)}$ is an $l$-element chain where $l=2$ or 4. Since $C_n$ embeds in infinitely many $C_m$, the coinitial sequence $(a_n)_{n<\omega}$ has infinitely many elements with the same label of $c_{n,t-1}$ where $c_{n,t-1}\in X$ is the singleton corresponding to $\alpha_{(1,t-1)}$. Select $l$ elements $j_1, \ldots, j_l$ of $J$ such that $j_1,\ldots,j_l<L_{n,t}$ where $L_{n,t}$ is the domain of $D_{n,t}$ and $r(j_s)=r_f(c_{n,t-1})$. Without loss of generality assume that $j_l < \cdots < j_1$. If $r(j_s)=-1$, resp $+1$, then, $P_{l1}:=P_{j_l}+\cdots+P_{j_1}$, resp $P_{1l}:=P_{j_1} + \cdots + P_{j_l}$, contains a chain of size $l$. Thus, $\alpha_{(1,t-1)}$ embeds in $P_{l1}$, resp $P_{1l}$. Similarly, for $\alpha_{(0,t-1)}$ which is an $l$-element antichain, choose $s_l < \cdots < s_1\in J$ with $s_1 < j_l$ and $r(s_k)=0$. Such elements exist since $I_0=\{i\in I : r(i)=0\}$ is coinitial in $I$. It is clear that $\alpha_{(0,t-1)}$ embeds in $P_{s_1}\oplus \cdots \oplus P_{s_l}$. Now, consider $D_{n,t-1}$. We have $D_{n,t-1}\leq C_n$. Since the set of $m$ such that $C_n$ embeds in $C_m$ is infinite, we can find $p$ such that $\sum (D_{n,t-1}+\cdots+D_{n,k_n})$ embeds in $\sum(C_{m+p}+\cdots+C_m)$. Then set $\varphi(t-1)=m+p$. This proves the induction step. Thus, $\sum D_n$ embeds in $\sum(C_{m+k}+\cdots+C_m)$ for some $k$.  

Now, let $n<\omega$ and suppose that the poset labelled sum $\sum(D_{n-1}+\cdots+D_0)$ has been embedded in $\sum(C_{\varphi(n-1)}+\cdots+C_0)$. By the argument above, $\sum D_n\leq \sum (C_{m+k}+\cdots+C_m)$ for some $m>\varphi(n-1)$ and $k$. Hence, $\sum(D_n+\cdots+D_0)\leq \sum(C_{m+k}+\cdots+C_0)$. Set $\varphi(n)=m+k$. This proves the second inequality.

For the last inequality, since $C_{\upharpoonright J}$ is left indecomposable, by Lemma \ref{ncopies} we have $\underline{2}.C_{\upharpoonright J}\leq C_{\upharpoonright J}$ which implies that $\sum \underline{2}.C_{\upharpoonright J}\leq \sum C_{\upharpoonright J}$. 
\end{proof}

Lemmas \ref{Nonisoplc} and \ref{ReducedequimorphP} imply the following. 

\begin{theorem}  \label{Reduced}
Let P be the sum of a countable labelled chain $C:=(I,\ell)$ where $I$ has no least element and $\ell(i)=(P_i,r(i))\in \mathcal{N}_{\leq\omega}\times\{-1, 0, +1\}$ such that r takes 0 and $\pm 1$ densely. Then  $Sib(P)=2^{\aleph_0}$.    
\end{theorem}

\subsection{Direct Sum}
 
Throughout this subsection $P$ is a countable direct sum of at least two non-empty connected $N$-free posets. Thus, it is disconnected.  
In this section we prove that $P$ has one or infinitely many siblings on condition that this property holds for each component of $P$. 

\begin{lemma} \label{Connecteddisconnected}
If some sibling of $P$ is  connected, then $Sib(P)=\infty$.
\end{lemma}

\begin{proof}
Let $P'\approx P$ where $P'$ is connected. So, $P'$ embeds into some component $Q$ of $P$. Since $P$ has at least two non-empty components, $P'\oplus 1\hookrightarrow Q\oplus 1\hookrightarrow P\hookrightarrow P'$. So, for every $n$, $P'\oplus \Bar{K}_n\hookrightarrow P'$ where $\Bar{K}_n$ is an antichain of size $n$. Since $P'$ is connected and $P'\approx P'\oplus \Bar{K}_n$, $P'$ and consequently $P$ has infinitely many pairwise non-isomorphic siblings. 
\end{proof}

\begin{lemma} \label{Infinitecomponent} 
If some component of P has infinitely many siblings, then \\ $Sib(P)=\infty$. 
\end{lemma}

\begin{proof}
Let $P$ satisfy the conditions of the lemma. Let $Q$ be a  component of $P$ with infinitely many pairwise non-isomorphic siblings $Q_1, Q_2, Q_3, \ldots$. For each $n$, let $P_n$ be the poset resulting from $P$ by replacing every component of $P$ which is equimorphic to $Q$ by $Q_n$. Now suppose that $P_m\cong P_n$ for $m\neq n$. Consider a component $Q'$ of $P_m$ equimorphic to $Q_m$. We know that $Q'$ is isomorphic to some component $Q''$ of $P_n$. We have $Q''\cong Q'\approx Q_m\approx Q$ which implies that $Q''=Q_n$ because $Q''$ is a component of $P_n$. But then $Q_m\cong Q_n$, a contradiction. 
Therefore, the resulting posets $P_n$ are pairwise non-isomorphic siblings of $P$. 
\end{proof}

\begin{lemma} \label{increasing}
Suppose that P has infinitely many non-trivial components. Then $Sib(P)=\infty$. 
\end{lemma}

\begin{proof}
Since $P$ has infinitely many non-trivial components and $\mathcal{N}_{\leq\omega}$ is w.q.o, there is an increasing sequence $(P_n)_{n<\omega}$  of non-trivial components of $P$. Let $\mathcal{Q}$ be the direct sum of the non-trivial components of $P$ other than the $P_n$. We have $P=\bigoplus_n P_n\oplus \mathcal{Q}\oplus A$ where $A$ is the direct sum of the trivial components of $P$. Notice that since $P$ is countable, so is $A$. Therefore, $A$ embeds in $\bigoplus_n P_{2n+1}$. Also $\bigoplus_n P_n$ embeds into $\bigoplus_n P_{2n}$. Hence, $P$ embeds into $P':=\bigoplus_n P_n\oplus \mathcal{Q}$. That is, $P'\approx P$. Similarly, $P'\oplus \Bar{K}_n \approx P'\approx P$ where $\Bar{K}_n$ is an antichain of size $n$. Since $P'\oplus\Bar{K}_n$ has exactly $n$ trivial components, it follows that $P$ has infinitely many pairwise non-isomorphic siblings. 
\end{proof}

Hence, by Lemmas  \ref{Connecteddisconnected}, \ref{Infinitecomponent} and \ref{increasing} we have the following. 

\begin{proposition} \label{InfinitesiblingnumberofP}
Let P be a countable direct sum of NE-free posets with at least two non-empty components. If P is a sibling of some connected NE-free poset, or some component of P has infinitely many siblings, or P has infinitely many non-trivial components, then $Sib(P)=\infty$. 
\end{proposition}

\begin{lemma} \label{Finitenontrivial} 
If  P has only finitely many non-trivial components and each component has only one  sibling, then $Sib(P)=1$.
\end{lemma}

\begin{proof}
Set $P:=\bigoplus_{i<l}P_i\oplus A$ where each $P_i$ is non-trivial and $A$ is the direct sum of the trivial components of $P$. If $l=0$, then $P$ is an antichain and $Sib(P) = 1$. Assume that $l>0$ and that $P' \subseteq P$ induces a sibling
of $P$ via an embedding $f$. We prove that $f$ induces a bijection on the set of indices of components of $P$. First note that since the components of $P$ are connected, for each $i$, there is $j$ such that $f(P_i)\subseteq P_j$. For each $i$, define $\hat{f}(i)=j$ where $j$ is such that $f(P_i)\subseteq P_j$. Suppose that for $i\neq j$, $\hat{f}(i)=\hat{f}(j)=k$. It follows that $f$ embeds $P_i\oplus P_j$ in $P_k$. We first show that $k$ cannot be $i$ or $j$. Suppose, without loss of generality, that $k=i$. Then $P_i\oplus 1\hookrightarrow P_i\oplus P_j\hookrightarrow P_i$ and by Lemma  \ref{Connecteddisconnected}, $P_i$ has infinitely many siblings, a contradiction. So, $k\neq i, j$. It follows that $\hat{f}(k)\neq k$ because otherwise $f$ embeds $P_i\oplus P_k$ in $P_k$ which implies that $P_k$ has infinitely many siblings, a contradiction. Further, $\hat{f}(k)\neq i, j$ because otherwise $P_i\oplus P_j\approx P_k$. By a similar argument, $\hat{f}^2(k)\neq i, j, k, \hat{f}(k)$. Iterating, the $P_{\hat{f}^n(k)}$ provide infinitely many non-trivial components  of $P$, a contradiction. Thus, $\hat{f}$ is injective on the set of indices of non-trivial components of $P$ and since there are finitely many such indices, $\hat{f}$ is bijective. Extending $\hat{f}$ to the set $I$ of indices of components of $P$, it follows that $\hat{f}$ is bijective on $I$. Pick $i<l$ and consider the orbit $\hat{f}.i$ of $i$ under $\hat{f}$. We have $P_i\approx P_j$ where $j\in \hat{f}.i$ and by assumption, $P_i\cong P_{\hat{f}.i}$ for each $i<l$. Hence $P'$ is isomorphic to $P$. 
\end{proof}

Now we are ready to conclude the following. 

\begin{proposition} \label{Directsumsiblings}
Let $P$ be a countable direct sum of at least two non-empty connected NE-free posets. If each component of P has one or infinitely many siblings, then $Sib(P)=1$ or $\infty$. 
\end{proposition}

\begin{proof}
If $P$ has some component with infinitely many siblings or $P$ has infinitely many non-trivial components, then $Sib(P)=\infty$ by Proposition \ref{InfinitesiblingnumberofP}. If $P$ has only finitely many non-trivial components each of which has only one sibling, then $Sib(P)=1$ by Lemma \ref{Finitenontrivial}. This completes the proof.
\end{proof}

\subsection{Linear Sum} \label{Linearsum}

By Theorem \ref{Pstructure}, we know that an $N$-free poset with more than one element which is not a direct sum or the sum of a labelled chain, is a linear sum $P$ of at least two non-empty $N$-free posets with property CCGC. In this section we prove that such a countable linear sum $P$ of $N$-free posets has one or infinitely many siblings on condition that this property holds for each summand of $P$. In this section, by a linear sum, we always mean a linear sum of $N$-free posets consisting of summands with property CCGC. 

\begin{lemma} \label{Uniquesummand}
Let $P=+_{i\in I} P_i$ and $Q=+_{j\in J} Q_j$ be two linear sums. Every embedding f from P to Q induces an order preserving map $\hat{f}:I\to J$.  Moreover, If f maps distinct summands of P into distinct summands of Q, resp f is an isomorphism, then $\hat{f}$ is a chain embedding, resp a chain isomorphism.  
\end{lemma}

\begin{proof}
We show that for every $i\in I$ there is $j\in J$ such that $f(P_i)\subseteq Q_j$. Let $i\in I$ be given. Pick $x, y\in P_i$. Set $G_i=CG(P_i)$. Since $G_i^c$ is connected, $x, y$ are connected in $G_i^c$ by a path $x=x_0, x_1, \ldots, x_{n-1}=y$. This means that $x_m\perp x_{m+1}$ in $P_i$ for every $0\leq m\leq n-2$. We have $f(x_m)\perp f(x_{m+1})$ for every $0\leq m\leq n-2$. Now suppose that $f(x)\in Q_j$ for some $j\in J$. Since the summands of $Q$ are comparable to each other, $f(x_m)\in Q_j$ for every $0\leq m\leq n-1$ because otherwise for two $m, m+1$, $f(x_m)$ and $f(x_{m+1})$ are comparable. This implies that $f(P_i)\subseteq Q_j$. Now define $\hat{f}:I\to J$ by $\hat{f}(i)=j$ where $j\in J$ is such that $f(P_i)\subseteq Q_j$. The mapping $\hat{f}$ is well-defined by above argument. If $i \leq_I i'$, then $P_i \leq_P P_{i'}$ which implies that $f(P_i)\leq_Q f(P_{i'})$ because $f$ is order-preserving. Consequently, $Q_j \leq_Q Q_{j'}$ where $f(P_i)\subseteq Q_j$ and $f(P_{i'})\subseteq Q_{j'}$. Hence, $\hat{f}(i) \leq_J \hat{f}(i')$. Moreover, if $f$ does not embed two distinct summands of $P$ into the same summand of $Q$, then $\hat{f}$ is injective meaning that it is a chain embedding. In particular, if $f$ is an isomorphism, then so is $f^{-1}$ and for every $i$ there is $j$ such that $f(P_i)\subseteq Q_j$ and $f^{-1}(Q_j)\subseteq P_i$ which imply that $f(P_i)=Q_j$. Thus, $\hat{f}$ defined as above is a chain isomorphism. 
\end{proof}

\begin{lemma} \label{InfiniteCproperty}
Let P be a linear sum. Suppose that some summand of P has infinitely many pairwise non-isomorphic siblings with property CCGC. Then P has infinitely many siblings.
\end{lemma} 

\begin{proof}
Set $P=+_{i\in I} P_i$.  Let $P_j$ be a summand  of $P$ such that  $\{P_{jn}\}_{n<\omega}$ is an infinite family of pairwise non-isomorphic siblings of $P_j$ with property CCGC. Let $J=\{i\in I : P_i\approx P_j\}$. For each $n<\omega$, set $P^n:=+_{i\in I} Q_i^n$ where $Q_i^n=P_i$ if $i\notin J$ and $Q_i^n=P_{jn}$ if $i\in J$. Notice that for each $i\in I$, $Q_i^n$ has property CCGC. Taking each summand $P_i$ to the summand $Q_i^n$ and vice versa, it is clear that $P^n\approx P$ for every $n<\omega$. Now suppose that $P^m\cong P^n$ for $m\neq n$ and let $f:P^m\to P^n$ be an isomorphism. Since $f$ is an isomorphism and both $P^m, P^n$ are linear sums whose summands have property CCGC, by Lemma \ref{Uniquesummand}, $f$ induces a chain isomorphism $\hat{f}:I\to I$ and for each $i\in I$, $f(Q_i^m)= Q^n_{\hat{f}(i)}$. More, we have $Q_i^m\cong Q_{\hat{f}(i)}^n$. Let $i\in J$. Then $P_{jm}= Q_i^m \cong Q_{\hat{f}(i)}^n = P_{jn}$, a contradiction. The last equality is due to the fact that $Q^n_{\hat{f}(i)}\approx P_{jm}\approx P_j$. This implies that the $P^n$ are pairwise non-isomorphic. Hence, $Sib(P)=\infty$. 
\end{proof}

The following proposition provides sufficient conditions to obtain infinitely many siblings for a countable linear sum. 

\begin{proposition} \label{InfiniteDproperty}
Let P be a countable linear sum. If some summand $P_j$ of P is the sum of a labelled chain with no least element; or $P_j$ is disconnected and has infinitely many siblings with property CCGC; or $P_j$ is disconnected and has at least one connected sibling, then P has infinitely many siblings. 
\end{proposition}

\begin{proof}
By Proposition \ref{Plsccgc} the sum of a countable labelled chain with no least element has property CCGC and by Theorem \ref{Reduced} it has continuum many siblings with the same properties. If $P$ contains a summand $P_j$ which is the sum of a labelled chain with no least element or $P_j$ is disconnected having  infinitely many siblings with property CCGC, then the statement is true by Lemma \ref{InfiniteCproperty}.

Assume that there is some  
disconnected summand $P_j$ of $P$ which is equimorphic to some connected $N$-free poset $Q$. Then $Q\approx P_j'$ for some connected component $P_j'$ of $P_j$. Consequently, the $\Bar{K}_n\oplus Q$, $n<\omega$, where $\Bar{K}_n$ is an antichain of size $n$, are pairwise non-isomorphic siblings of $Q$ and thus of $P_j$ which have property CCGC. Hence, in this case $P$ has infinitely many siblings by Lemma \ref{InfiniteCproperty}. 
\end{proof}

Let $P$ be a countable linear sum with more than one element such that all its summands have only one sibling. Then $P$ does not contain a summand which is the sum of a labelled chain with no least element. That is, each summand of $P$ is either a singleton or a disconnected $N$-free poset. Obviously, $P$ does not embed in a singleton and if $P$ embeds in a disconnected summand $P_j$, then $P_j$ has infinitely many siblings by Lemma \ref{Connecteddisconnected} which contradicts our assumption. In other words, each summand of $P$ strictly embeds into $P$. In case a linear sum has only finitely many non-trivial summands having only one sibling, we show that $P$ has one or infinitely many siblings. However, in light of the following result, we prove a stronger statement.

\begin{theorem} [\cite{LPW},  Corollary 3.6] \label{Dichchain}
If C is a countable chain, then $Sib(C)=1$, $\aleph_0$ or $2^{\aleph_0}$.  
\end{theorem}

\begin{lemma} \label{Finitesummand}
Let $P=+_{i\in I} P_i$ be a countable linear sum with only finitely many non-trivial summands each of which has only one sibling. Then $Sib(P)=1$ or $\aleph_0$ or $2^{\aleph_0}$. 
\end{lemma}

\begin{proof}
Let $f$ be an embedding of $P$ and $\hat{f}$ be defined as in Lemma \ref{Uniquesummand}. Let $K$ be the chain of indices $i\in I$ such that $P_i$ is  non-trivial. If $K=\emptyset$, then $P=I$  and there is nothing to prove (the statement holds for a countable chain by Theorem \ref{Dichchain}). Assume that $K\neq\emptyset$. We know that $\hat{f}$ is order-preserving. We show that $\hat{f}$ is injective. For the sake of a contradiction, assume that $\hat{f}(i)=\hat{f}(j)=k$ for some $i < j\in I$. It implies that $k\in K$. Since $P_k$ is non-trivial, $\hat{f}(k), \hat{f}^2(k), \ldots\in K$. Since $K$ is a finite chain and $\hat{f}$ is order-preserving, we conclude that for some $l\in K$ and some integer $m$, $\hat{f}^m(k)=\hat{f}^{m+1}(k)=l$. Without loss of generality, assume that $i\leq_I l$. Let $\Bar{I}=\{i'\in I : i\leq_I i' \leq_I l\}$. We have $\hat{f}^{m+1}(\Bar{I})=l$. Now consider $P'=+_{i' \in \Bar{I}} P_{i'}$. From $\hat{f}^{m+1}(\Bar{I})=l$, we conclude that $f^{m+1}$ embeds $P'$ into $P_l$. Clearly, $P_l$ embeds into $P'$. That is $P'\approx P_l$. Since $l\in K$, $P_l$ is disconnected. Also, $P_l$ is equimorphic to a connected $N$-free poset $P'$ which implies that $Sib(P_l)=\infty$ by Lemma \ref{Connecteddisconnected}, a contradiction. Therefore, $\hat{f}$ is a chain embedding of $I$. Further, since $K$ is finite, $\hat{f}$ is the identity on $K$. Consequently, $f(P_k)\subseteq P_k$ for each $k\in K$.

Set $K:=\underline{n}$. We may represent $P$ as $P=I_0+P_0+\cdots+P_{n-1}+I_n$ where the $P_k$ are the non-trivial summands of $P$, $I_0=\{x\in P : \{x\} <_P P_0\}$,  $I_k=\{x\in P : P_{k-1} <_P \{x\} <_P P_k\}$ for every $1\leq k\leq n-1$ and $I_n=\{x\in P : P_{n-1} <_P \{x\}\}$. Notice that each $I_k$ is a countable chain. By the above argument, each sibling of $P$ is of the form $Q=J_0+Q_0+\cdots+Q_{n-1}+J_n$ where $J_k\approx I_k$, $k\in\underline{n+1}$, and $Q_k\approx P_k$, $k\in\underline{n}$.  Since $Sib(P_k)=1$, we have $Q_k\cong P_k$. Moreover, each countable chain has one or countably many or continuum many siblings by Theorem \ref{Dichchain}.  Hence, $Sib(P)=\max\{Sib(I_k) : k\in\underline{n+1}\}$. It follows that $Sib(P)=1$ or $\aleph_0$ or $2^{\aleph_0}$.  
\end{proof}

It remains to verify the case in which a linear sum contains infinitely many non-trivial summands, each having only one sibling. However, we will show that when a linear sum has infinitely many non-trivial summands, it has continuum many siblings, regardless of the sibling number of its summands. 
We may consider a linear sum $P=+_{i\in I}P_i$ as the sum $\sum C$ of a labelled chain $C=(I,\ell)$ where $\ell:I\to \mathcal{N}_{\leq\omega}\times \{-1\}$ is defined by $\ell(i)=(P_i,r(i))$ and $r(i)=-1$ for each $i\in I$. Since the map $r$ is constant, we may abuse the notation and write $\ell(i)=P_i$. Note that the labels belong to a b.q.o. We denote by $K$ the chain of $i\in I$ such that $P_i$ is non-trivial. Our proof contains two main ingredients: Lemmas \ref{ConstructionLS} and \ref{Noleastlinearcontinuum}. 

Let $n\geq 1$ and consider the chains $\underline{n}$ and $\underline{n+2}$. We denote by $A_n$, resp $B_n$, the poset obtained by replacing each $m\in \underline{n}$, resp $0 < m < n+1$, with a two-element antichain. Note that each $A_n$ and $B_n$ is a linear sum whose summands have property CCGC with context chains $\underline{n}$ and $\underline{n+2}$, respectively.  
Two elements $i$ and $j$ of $K$ are equivalent, denoted by $i\equiv_K j$, just in case every element of the interval in $I$ with endpoints $i$ and $j$ is in $K$. It is clear that $\equiv_K$ is an equivalence relation. We denote by $\mathcal{K}$ the set of equivalence classes of $K$. The equivalence classes of $K$ are maximal intervals of $I$ that are contained in $K$. 

\begin{lemma} \label{ConstructionLS}
Let $P=+_{i\in I} P_i$ be a countable linear sum such that K is coinitial in I and without least element. Then, there are $2^{\aleph_0}$ pairwise non-isomorphic linear sums $P_f=\sum C_f$ where $f\in \{0,1\}^\mathbb{N}$. 
\end{lemma}

\begin{proof}
It follows that $I$ has no least element. Set $C:=(I,\ell)$ where $\ell(i)=P_i$. 
Recall that in this case, by Lemma \ref{Labelledchainsproperties} (1), there is an initial segment $J$ of $I$ such that $C_{\upharpoonright J}$ is left indecomposable. For each $K'\subset J$ with $K'\in\mathcal{K}$, if $|K'|=2$ or $|K'|=4$, then insert a new element $d$ such that $d$ covers the greatest element of $K'$ and set $\ell'(d):=A_1$. Thus, $C_{\upharpoonright K'}$ is replaced with a  labelled chain which is isomorphic to $A_3$ or $A_5$. Let $Y$ be the collection of all these new elements $d$. Set $J':=J\cup Y$.   
We select an infinite descending sequence $(a_n)_{n<\omega}$ of elements of $K\cap J$ coinitial in $J$. To each map $f:\mathbb{N}\to\{0,1\}$ we associate a  labelled chain $C_f$ as follows. Let $f:\mathbb{N}\to\{0,1\}$. For $i\in I\setminus Y$ set $\ell_f(i):=\ell(i)$ and for $i\in Y$ set $\ell_f(i):=\ell'(i)$.  For each $n<\omega$, set $k(n):=2f(n)+4$ and insert elements $b_{1,n} < \cdots < b_{k(n),n}$ such that $b_{1,n}$ covers $a_n$ and $b_{i+1,n}$ covers $b_{i,n}$ for each $1\leq i\leq k(n)-1$. Set $\ell_f(b_{1,n}):=b_{1,n}$, $\ell_f(b_{k(n),n}):=b_{k(n),n}$ and $\ell_f(b_{i,n}):=A_1$ where $1<i<k(n)$. Indeed, the chain $b_{1,n}<\cdots<b_{k(n),n}$ is the context chain of $B_{2f(n)+2}$. Let $b_n$ be the chain $b_{1,n}<\cdots<b_{k(n),n}$. Then, $(b_n)_{n<\omega}$ is a sequence of chains such that $b_{n+1}<b_n$.  Set $X:=\bigcup_{n<\omega} b_n$ which is coinitial in $J$. Set $\Bar{J}:=J'\cup X$, $F:=I\setminus \Bar{J}$ and $\Bar{I}:=F\cup \Bar{J}$. Then, $C_f:=(\Bar{I},\ell_f)$ is a labelled chain. Thus, we obtain a sum $\sum C_f$ which is a linear sum whose summands have property CCGC. That is, $\sum C_f=+_{i\in\Bar{I}}P_i$ such that $K$ has no least element and $K$ is coinitial in $\Bar{I}$.

Suppose that $f, g$ are two maps from $\mathbb{N}$ to $\{0,1\}$ such that $\sum C_f\cong \sum C_g$ by some isomorphism $\phi$. Since the summands of these linear sums have property CCGC,  by Lemma \ref{Uniquesummand}, $\phi$ induces an isomorphism $h$ on $\Bar{I}$ preserving the labels, that is $\ell_f(i)=\ell_g(h(i))$ for each $i\in \Bar{I}$. If for each $n<\omega$, $h(b_n)\subset F$, then $h(\Bar{I})\subseteq F$ because $X$ is coinitial in $\Bar{I}$. But this means that $h$ is not surjective, a contradiction. Therefore, $h(b_n)\cap \Bar{J}\neq\emptyset$ for some $n$. Since each $b_n$ is the context chain of $B_2$ or $B_4$ and since there is no  labelled chain $C_{\upharpoonright K'}$, $K'\in\mathcal{K}$, in $C_{g\upharpoonright\Bar{J}}$ which is isomorphic to $A_2$ or $A_4$, it follows that $h(b_n)=b_m$ for some $m$ and consequently, for each $k\geq 1$, $h(b_{n+k})=b_{m+k}$. Thus, $h$ may be regarded as an isomorphism from a final segment $D$ of $\mathbb{N}$ onto another $D'$ such that $g(h(n))=f(n)$ for each $n\in D$. By Lemma \ref{Continuummaps}, there are $2^{\aleph_0}$ maps in $\{0,1\}^\mathbb{N}$ such that for two distinct maps $f$ and $g$, we have $\sum C_f\ncong \sum C_g$. For each $f\in\{0,1\}^\mathbb{N}$, set $P_f:=\sum C_f$. Then, the set $\{P_f\}_{f\in\{0,1\}^\mathbb{N}}$ has the desired properties. 
\end{proof}

In lemma \ref{ConstructionLS}, we observe that the order on $P_f:=\sum C_f$ extends both $\sum C_{f\upharpoonright\Bar{J}}$ and $\sum C_{\upharpoonright F}$ as follows: if $x\in\sum C_{f\upharpoonright\Bar{J}}$ and $y\in \sum C_{\upharpoonright F}$, then $x\in P_i$ and $y\in P_j$ for some $i\in \Bar{J}$ and $j\in F$. We have $i<j$ and since $r(i)=-1$, it follows that $x <_{P_f} y$. Therefore,  $\sum C_f =\sum C_{f\upharpoonright\Bar{J}}+\sum C_{\upharpoonright F}$.

In order to show that the $\sum C_f$ in Lemma \ref{ConstructionLS} are embeddable into $\sum C$, first we prove the following claim.

\begin{claim} \label{J'embedsJ}
We have $C_{\upharpoonright J'}\leq C_{\upharpoonright J}$. Moreover, $C_{\upharpoonright J'}$ is left indecomposable.
\end{claim}

\begin{proof}
We show that the following inequalities hold.
$$ C_{\upharpoonright J'}\leq  \underline{2}.C_{\upharpoonright J}\leq  C_{\upharpoonright J}.$$

For the first inequality, recall $Y$ in Lemma \ref{ConstructionLS}  which is countable. Let $d_0, \ldots, d_n, \ldots$ be an enumeration of $Y$. For each $n$, let $c_n\in J$ be such that $d_n$ covers $c_n$. Therefore, $c_n\in K$. Define an embedding $\varphi:C_{\upharpoonright J'}\to\underline{2}.C_{\upharpoonright J}$ as follows. Define $f:J'\to\underline{2}.J$ by $f(j)=(0,j)$ if $j\in J$ and $f(d_n)=(1,c_n)$ for each $n$. Since $\ell'(d_n)$ is a 2-element antichain and $\ell((1,c_n))=\ell(c_n)$ is a direct sum of at least two non-empty posets, it follows that $\ell'(d_n)$ embeds in $\ell(f(d_n))$. Therefore, $C_{\upharpoonright J'}\leq \underline{2}.C_{\upharpoonright J}$. 

The second inequality is because $\underline{2}.C_{\upharpoonright J}\leq C_{\upharpoonright J}$, a fact which follows from Lemma \ref{ncopies} since $C_{\upharpoonright J}$ is left indecomposable.  

We know that $C_{\upharpoonright J}$ is left indecomposable. Let $L'+R'$ be a partition of $J'$ into an initial and a final segment and set $L:=L'\setminus Y$ and $R:=R'\setminus Y$. Then $L+R$ is a partition of $J$ into an initial and a final segment. By the inequality above and the fact that $C_{\upharpoonright J}$ is left indecomposable, we have $C_{\upharpoonright J'}\leq C_{\upharpoonright J}\leq C_{\upharpoonright L}\leq C_{\upharpoonright L'}$. Thus, $C_{\upharpoonright J'}$ is left indecomposable. 
\end{proof}

\begin{lemma} \label{Noleastlinearcontinuum}
Let $P=+_{i\in I}P_i$ be a countable linear sum such that K is coinitial in I and has no least element. Then $Sib(P)=2^{\aleph_0}$. 
\end{lemma}

\begin{proof}
Let $f:\mathbb{N}\to\{0,1\}$ and $C_f$ be the labelled chain constructed in Lemma \ref{ConstructionLS}. We prove that $\sum C_f\leq \sum C$. We have $C_f= C_{f\upharpoonright \Bar{J}}+C_{\upharpoonright F}$ because $f$ does not change $F$. By the observation given after Lemma \ref{ConstructionLS} we have $\sum C_f =\sum C_{f\upharpoonright\Bar{J}}+\sum C_{\upharpoonright F}$. It suffices to prove that $\sum C_{f\upharpoonright \Bar{J}}\leq \sum C_{\upharpoonright J}$.  We prove that 
$\sum C_{f\upharpoonright \Bar{J}}\leq \sum C_{\upharpoonright J'}$ because then the result follows from Claim \ref{J'embedsJ}. 

By Claim \ref{J'embedsJ}, $C_{\upharpoonright J'}$ is left indecomposable, so, we can write it as $\sum^*_{n<\omega} C_n$ where each $C_n$ is indecomposable and the set of $m$ such that $C_n\leq C_m$ is infinite. 
For each $n$, let $J_n$ be the domain of $C_n$. Set $L_n:=J_n\cup\bigcup_{a_m\in J_n} b_m$ and $D_n:=(L_n,\ell_{f\upharpoonright L_n})$. Then we have $C_{f\upharpoonright \Bar{J}}=\sum^*_{n<\omega} D_n$. We define an embedding from $\sum C_{f\upharpoonright \Bar{J}}$ into $\sum C_{\upharpoonright J'}$ by induction. Let $n$ be given. Suppose that $\sum(D_{n-1}+\cdots+D_0)$ has been embedded into $\sum(C_{\varphi(n-1)}+\cdots+C_0)$. The labelled chain $D_n$ consists of $C_n$ plus finitely many elements, each labelled by a singleton or a 2-element antichain. Thus, $D_n$ can be written
$D_{n,0}+\alpha_{n,0}+D_{n,1}+\cdots+\alpha_{n,k_n-1}+D_{n,k_n}$
with $D_{n,0}+D_{n,1}+\cdots+D_{n,k_n}=C_n$ and each $\alpha_{n,i}$ is a finite chain whose elements are labelled by either a singleton or an antichain of size 2. Suppose that $\sum(D_{n,i}+\alpha_{n,i}+\cdots+\alpha_{n,k_n-1}+D_{n,k_n})$ has been embedded into $\sum(C_{m+k}+\cdots+C_m)$ for some $m > \varphi(n-1)$ and some $k$. We know that the sequence $(a_n)_{n<\omega}$ is coinitial in $J'$ and the $P_{a_n}$ are non-trivial. Pick a chain $a_{n_i}<\cdots<a_{n_0}< J_{m+k}$ of elements of $(a_n)_{n<\omega}$ such that its length equals to the size of the context chain of $\alpha_{n,i-1}$.  Then, $\alpha_{n,i-1}$ embeds in $P_{a_{n_i}}+\cdots+P_{a_{n_0}}$.  We have $D_{n,i-1}\leq C_n$. Further, $C_n$ embeds in infinitely many $C_m$. Therefore, we can find sufficiently large $l$ such that $J_l < a_{n_i}$ and $D_{n,i-1}$ embeds in $C_l$. Hence, $\sum D_n$ embeds in $\sum(C_{m+p}+\cdots+C_m)$ for some $p$. Then by setting $\varphi(n)=m+p$ we have $\sum(D_n+\cdots+D_0)\leq \sum(C_{\varphi(n)}+\cdots+C_0)$. This completes the induction step as well as the proof.
\end{proof}

Let $P$ be a poset. It is easy to verify that $P$ and $P^*$ have the same sibling number.  With this observation, if in a countable linear sum $P=+_{i\in I} P_i$, the chain $K$ has no greatest element and it is cofinal in $I$, then we may consider $P^*$ in which $K^*$ has no least element and it is coinitial in $I^*$. Then, by applying Lemmas \ref{ConstructionLS} and \ref{Noleastlinearcontinuum}, we get continuum many siblings for $P^*$ and thus for $P$. 

In both Lemmas \ref{ConstructionLS} and \ref{Noleastlinearcontinuum}, $K$ has no least (greatest) element and it is coinitial (cofinal) in $I$. This is not necessarily the case in a general linear sum. However, the following lemma implies that we can always find some interval $J$ of $I$ in which $K$ is coinitial (cofinal) in $J$ with no least (greatest) element.  

\begin{lemma} \label{Interval}
Let I be an infinite chain. Then there exists an interval J of I such that J has no least or greatest element. 
\end{lemma}

\begin{proof}
For the sake of a contradiction, assume that for each interval $J$ of $I$, $J$ has least and greatest element. Set $I_0:=I$ and let $x_0$, resp $y_0$, be the least, greatest, element of $I_0$. For $n\geq 1$, set $I_{n+1}:=I_n\setminus\{x_n,y_n\}$ and let $x_{n+1}$, resp $y_{n+1}$, be the least, resp greatest, element of $I_{n+1}$. Since $I$ is infinite, $I_n\neq\emptyset$ for every $n<\omega$. Then, the initial interval $J:=\{x\in I : \exists n, x < x_n\}$ has no greatest element, a contradiction. 
\end{proof}

\begin{proposition} \label{Infinitesummand}
Let $P=+_{i\in I} P_i$ be a countable linear sum with infinitely many non-trivial summands. Then, P has continuum many siblings. 
\end{proposition}

\begin{proof}
Let $K$ be the chain of indices $i\in I$ such that $P_i$ is non-trivial. By assumption, $K$ is countably infinite. Let $I_1$, resp $I_2$, be the maximal interval of $I$ such that $I_1<_I K$, resp $K<_I I_2$. Then $I=I_1+\hat{I}+I_2$ where $\hat{I}$ is the minimal interval of $I$ containing $K$. By Lemma \ref{Interval}, without loss of generality, there exists a final interval $J$ of $\hat{I}$ and a final interval $\hat{K}$ of $K$ such that $\hat{K}$ has no least element and it is coinitial in $J$. Thus, if we set 
$$J:=\{i\in\hat{I} : \exists (k_n)_{n<\omega},\ \forall n,\  k_n\in K\ \text{and}\ k_{n+1} < k_n < i\},$$
then $J\neq\emptyset$ and it is maximal among all final intervals of $\hat{I}$ in which there is a descending sequence of elements of $K$ with no least element. Set $L:=\hat{I}\setminus J$. Then, $L\cap K$ is well-ordered because otherwise there is a descending sequence of elements of $K$ in $L$ which contradicts the maximality of $J$. We have $\hat{I}=L+J$ and also $I=I_1+L+J+I_2$. By Lemmas \ref{ConstructionLS} and \ref{Noleastlinearcontinuum}, the linear sum $Q:=+_{j\in J} P_j$ has $2^{\aleph_0}$ pairwise non-isomorphic siblings $\{Q_\alpha\}_{\alpha < \mathfrak{c}}$ where $\mathfrak{c}$ is the continuum. For each $\alpha<\mathfrak{c}$ set $\hat{Q}_\alpha:=Q_\alpha+(+_{i\in I_2} P_i)$. Suppose that $\hat{Q}_\alpha\cong \hat{Q}_\beta$ for two $\alpha\neq \beta$ by some isomorphism $\phi$. Since the $\hat{Q}_\alpha$ are linear sums whose summands have property CCGC, by Lemma \ref{Uniquesummand}, $\phi$ induces an isomorphism $h$ on $J+I_2$. If for some $x\in J$, $h(x)\in I_2$, then $x\notin K$. Since by assumption there is some $y\in J\cap K$ with $x<_I y$, we have $h(x)<_I h(y)$. But then $h(y)$ must be the index of a non-trivial summand in $+_{i\in I_2} P_i$ which is not possible. Therefore, $h(J)=J$ meaning that $Q_\alpha\cong Q_\beta$, a contradiction. Hence, the $\hat{Q}_\alpha$ are pairwise non-isomorphic. Next, for each $\alpha<\mathfrak{c}$ set $\hat{P}_\alpha:=(+_{i\in L}P_i)+\hat{Q}_\alpha$. Suppose that $\hat{P}_\alpha\cong \hat{P}_\beta$ for two $\alpha\neq\beta$ by an isomorphism $\phi$. By Lemma \ref{Uniquesummand}, $\phi$ induces an isomorphism $h$ on $L+J+I_2$. Suppose that for some $x\in J$, $h(x)\in L$. By definition of $J$, select a descending sequence $(k_n)_{n<\omega}$ of elements of $K$ which has no least element in $J$ and $k_n < x$ for each $n$. Then $(h(k_n))_{n<\omega}$ is a coinitial sequence of elements of $K$ in $L$ with no least element. But this is impossible because $L\cap K$ is well-ordered. Thus, $h(J+I_2)=J+I_2$ meaning that $\hat{Q}_\alpha\cong \hat{Q}_\beta$, a contradiction. Finally, for each $\alpha<\mathfrak{c}$ set $P_\alpha:=(+_{i\in I_1} P_i)+\hat{P}_\alpha$. By an argument similar to what we gave for the $\hat{Q}_\alpha$, it is proven that the $P_\alpha$ are pairwise non-isomorphic siblings of $P$. 
\end{proof}

Now we are ready to conclude the main result of this section. 

\begin{proposition} \label{Linearsumsiblings}
Let P be a countable linear sum of at least two non-empty NE-free posets with property CCGC. If each summand has one or infinitely many siblings, then, so does P. 
\end{proposition}

\begin{proof}
Let $P=+_{i\in I} P_i$ be a countable linear sum. 

\noindent\textbf{Case 1} Some summand $P_j$ has infinitely many siblings. Then either $P_j$ is the sum of a  labelled chain with no least element or $P_j$ is disconnected and has infinitely many siblings with property CCGC. In this case $Sib(P)=\infty$ by Proposition \ref{InfiniteDproperty}.  

\noindent\textbf{Case 2} All summands of $P$ have only one sibling. 

\noindent\textbf{Subcase 2.1} $P$ has only finitely many non-trivial summands. In this case $Sib(P)=1$ or $\aleph_0$ or $2^{\aleph_0}$ by Lemma \ref{Finitesummand}. 

\noindent\textbf{Subcase 2.2} $P$ has infinitely many non-trivial summands. In this case $Sib(P)=2^{\aleph_0}$ by Proposition \ref{Infinitesummand}. 
\end{proof}

\subsection{Proof of Theorem \ref{AltThomasseNEP}}

In this short section, we prove that the alternate Thomass\'{e} conjecture holds for a countable $N$-free poset $P$, that is $P$ has one or infinitely many siblings.

\begin{proof} (of Theorem \ref{AltThomasseNEP})
Let $P$ be a countable $N$-free poset. Note that the class of countable $N$-free posets is w.q.o under embeddability and thus well-founded. Therefore, we can use induction. For each countable $N$-free poset $Q$, let $\mathcal{S}(Q)$ be the statement ``$Sib(Q)=1$ or $\infty$". The statement is true for a singleton. Let $P$ have more than one element and assume that $\mathcal{S}(Q)$ holds for each $Q$ strictly embedding in $P$. By Theorem \ref{Pstructure}, $P$ is either (1) a direct sum of at least two non-empty connected $N$-free posets, or (2) a linear sum of at least two non-empty $N$-free posets with property CCGC, or (3) the sum $\sum C$ of a  labelled chain $C=(I,\ell)$ with no least element such that each $\ell(i)$ is the pair $(P_i,r(i))$ made of an $N$-free poset $P_i$ and  $r(i)\in\{-1, 0, +1\}$ where $r$ takes 0 and $\pm 1$ densely. For case (3), by Theorem \ref{Reduced} we have $Sib(P)=2^{\aleph_0}$. Assume that $P$ satisfies (1), resp (2). Then, $P$ is disconnected, resp connected. If $P$ embeds in some of its components, resp disconnected summands, then $P$ is equimorphic to a connected, resp disconnected, $N$-free poset and $Sib(P)=\infty$ by Lemma \ref{Connecteddisconnected}. Otherwise, each component, resp summand, of $P$ strictly embeds in $P$. By induction hypothesis and Proposition \ref{Directsumsiblings}, resp Proposition \ref{Linearsumsiblings}, $\mathcal{S}(P)$ holds. 
\end{proof}

\section{Finite Poset Substitution of Chains or Antichains}  \label{Finiteposetsub}

In this section, we prove that a countable $N$-free poset whose comparability graph has one sibling, is a finite poset substitution of chains or antichains and that it has one, countably many or else continuum many siblings. Moreover, the comparability graph of $N$-free posets of this section is a finite graph substitution of cliques or independent sets. Thus, the result of this section is a generalisation of Thomass\'e's conjecture for countable chains whose comparability graph is a clique.

The following theorem proven in \cite{HPW} describes the structure of countable cographs $G$ with only one sibling which will be used in determining the structure of a poset $P$ with $CG(P)=G$.  

\begin{theorem} [\cite{HPW}] \label{Lexico}
Let $G$ be a countable cograph. $G$ has only one sibling if and only if $G$ is a finite graph substitution of cliques or independent sets.
\end{theorem}  

Let $P$ be a countable $N$-free poset such that its comparability graph $G$ has one sibling. Then, by Theorem \ref{Lexico}, $G=K[H_v/v : v\in K]$ where $K$ is a finite graph and each  $H_v$ is a clique or an independent set. We know that $G$ is a cograph, thus,
by Proposition \ref{Subcograph}, so is $K$. Throughout this section let $T:=T(K)$ be the decomposition tree  of $K$. For every $M\in T$, denote by $P(M)$ the poset obtained by restricting $P$ to the graph blocks $H_v$ of $G$ where $v\in M$. For every leaf $\{u\}$ of $T$ we write $P_u$ instead of $P(\{u\})$. We verify the following statement by induction on the vertices of $T$. For every $M\in T$, let 
\begin{center} 
     $\mathcal{S}_{fps}(M)$ : $P(M)$ is a finite poset substitution of chains or antichains. 
\end{center}
The statement $\mathcal{S}_{fps}(M)$  also provides a structure of the poset $P(M)$ which will be used in the proof of the following statement: 
\begin{center}
    $\mathcal{S}_{sib}(M)$ : $Sib(P(M))=1$ or $\aleph_0$ or $2^{\aleph_0}$.
\end{center}
Since $T$ is finite, this proves $\mathcal{S}_{fps}(M)$ and $\mathcal{S}_{sib}(M)$ for all $M\in T$. In particular, $P$ has one, countably many or continuum many siblings.

We say that a node $M$ of $T$ is \textit{edge minimal} if $M$ does not have a module $\{u, v\}$ of $K$ satisfying the following:   
\begin{enumerate}
    \item $u$ and $v$ are adjacent and both $H_u$ and $H_v$ are cliques;
    \item $u$ and $v$ are non-adjacent and both $H_u$ and $H_v$ are independent sets.  
\end{enumerate} 
We call $T$ \textit{leaf minimal} if every node of $T$ is edge minimal. It follows that if $M\in T$ is edge minimal, then for every $u, v\in M$ with property (1) or (2), $\{u,v\}$ is not a module of $K$, thus, there exists some $w\in K$ adjacent to one and only one of $u$ or $v$. Since $M$ is a module of $K$, we conclude that $w\in M$ meaning that $\{u,v\}$ is not a module of $M$. Therefore, $M$ does not have two children $\{u\}, \{v\}$ as leaves of $T$ such that $u, v$ satisfy (1) or (2). 
We say that the context graph $K$ of $G=K[H_v/v : v\in K]$ where the graph blocks are cliques or independent sets, is the \textit{smallest} when $G$ cannot be represented as $K'[H'_v/v : v\in K']$ where $K'$ is a graph,  $|K'|<|K|$ and each $H'_v$ is a clique or an independent set. By the following lemma, it turns out that the smallest possible context graph $K$ for $G$ implies that $T$ is leaf minimal.

\begin{lemma} \label{Kminimal} 
Let $G=K[H_v/v : v\in K]$ where $K$ is a finite cograph and each $H_v$ is either a clique or an independent set and $T$ be the decomposition tree of $K$. If $K$ is the smallest, then $T$ is leaf minimal. 
\end{lemma}

\begin{proof} 
Suppose $T$ is not leaf minimal. Then some vertex $M\in T$ is not edge minimal. Let $u, v\in M$ form a module such that either they are adjacent and both $H_u, H_v$ are cliques or they do not form an edge and both $H_u, H_v$ are independent sets. In either case let $H_a$ be the graph obtained by restricting $G$ to the graph blocks $H_u$ and $H_v$. We know that in the first, resp second, case, $H_a$ is a clique, resp an independent set. Also, since $N:=\{u, v\}$ is a module, for $w\in K\setminus N$,  we have $uw\in E(K)$ if and only if $vw\in E(K)$. Let $K'$ be a graph whose vertex set is $V':=(K\setminus N)\cup \{a\}$ and its edge set is obtained by replacing all $wu, wv\in E(K)$ with $wa$ where $w\in K\setminus N$. Let $x\in G$. Then $x\in H_w$ for some $w\in K$. If $w\in K\setminus N$, then $w\in V'$ and $x\in H_w$; if $w=u$ or $w=v$, then $x\in H_a$ and $a\in V'$. Now, let $xy\in E(G)$. Then either $x, y\in H_w$, $w\in K$, and $xy\in E(H_w)$; or $x\in H_w$, $y\in H_{w'}$, $w, w'\in K$ and $ww'\in E(K)$. In the first case, if $w\in K\setminus N$, then $w\in V'$ and $xy\in E(H_w)$; and if $w=u$ ($w=v$), then $H_u$ ($H_v$) is a clique and we have $xy\in E(H_a)$. In the second case when $w, w'\in K\setminus N$, then $w, w'\in V'$ and $ww'\in E(K')$; if $w\in K\setminus N$ and $w'\in N$, then $wa\in E(K')$; and finally if $w,w'\in N$, then $x, y \in H_a$ and $uv\in E(K)$ which means that $xy\in E(H_a)$.  Hence, $G=K'[H'_v/v : v\in K']$. Since $K'$ is a graph, $|K'|<|K|$ and each $H_w$, $w\in K'$, is a clique or independent set, we get a contradiction because $K$ is the smallest context graph of $G$.     
\end{proof}

We say that a node $M\in T$ is \textit{gs-connected}, resp \textit{gs-disconnected}, if  $G(M)=M[H_v/v : v\in M]$ is connected, resp disconnected. Note that if $v(M)=1$, then each pair $\{u, v\}$ of elements of $M$ forms an edge meaning that $M$ is gs-connected. Similarly, if $v(M)=0$, then $M$ contains two non-adjacent elements $u, v$ meaning that $M$ is gs-disconnected.  So, when $M$ is gs-disconnected, then $P(M)$ has at least two components. 
We will use the edge minimality of the nodes of $T$ to argue by induction on those vertices.

\begin{lemma} \label{Disconnectedmodule}
Let $M\in T$ be such that $v(M)=1$ and $N$ a gs-disconnected child of $M$.
\begin{enumerate}
    \item If for some $y\in P(M\setminus N)$ and some $x\in P(N)$ we have $x <_P y$, resp $y <_P x$, then $P(N) <_P \{y\}$, resp $\{y\} <_P P(N)$. 
    \item If $N'$ is another gs-disconnected child of $M$, then we have $P(N) <_P P(N')$ or $P(N') <_P P(N)$. 
\end{enumerate}
\end{lemma}

\begin{proof}
(1) 
Suppose some $y\in P(M\setminus N)$ and some $x\in P(N)$ satisfy $x <_P y$. Since $N$ is gs-disconnected, there is some $z\in P(N)$ such that $x$ and $z$ belong to different components of $P(N)$ that is $x \perp z$. Also, $z\in H_u$ and $y\in H_v$ for some $u, v\in M$. Since $v(M)=1$,  $u$ and $v$ are adjacent. Thus, $zy\in E(G(M))$. In other words, $z$ and $y$ are comparable w.r.t $\leq_P$. We have $z <_P y$ because otherwise we get $x <_P y <_P z$ which yields that $x$ and $z$ are comparable, a contradiction. It follows that for every element $z\in P(N)$ in a component other than the component to which $x$ belongs, we have $z <_P y$. By a similar argument, for every element  $x\in P(N)$ in a component other than the component to which $z$ belongs, we have $x <_P y$. Hence, $P(N) <_P \{y\}$ in this case. Similarly, if there is some $x\in P(N)$ and $y\in P(M\setminus N)$ satisfying  $y <_P x$, then $\{y\} <_P P(N)$. 

(2) Now let $N'$ be another gs-disconnected child of $M$. Since $v(M)=1$, every element of $P(N)$ is comparable to every element of $P(N')$. Take some $x\in P(N)$ and suppose that for some $y\in P(N')$ we have $x <_P y$. Since $N'$ is gs-disconnected, (1) implies that $\{x\} <_P P(N')$. Since $N$ is gs-disconnected, take some $z\in P(N)$ such that $x$ and $z$ belong to different components of $P(N)$. We have $z\perp x$. Therefore, $\{z\} <_P P(N')$ because otherwise we get $x <_P y <_P z$, a contradiction. So, for every $z\in P(N)$ in a component of $P(N)$ other than the component to which $x$ belongs, we have $\{z\} <_P P(N')$. Similarly, for every $x$ in a component of $P(N)$ other than the component $z$ belongs to, we have $\{x\} <_P P(N')$. It follows that $P(N) <_P P(N')$. The argument for the other case is similar.    
\end{proof}

With the above lemma at our disposal, we can prove $\mathcal{S}_{fps}(K)$. The proof is by induction on the vertices of $T$.

\begin{proposition} \label{Finiteps} 
Let $P$ be a countable $N$-free poset whose comparability graph is a finite graph substitution of cliques or independent sets. Then,  $P$ is a finite poset substitution of chains or antichains. 
\end{proposition}

\begin{proof}
Suppose $G:=CG(P)$ is of the form $K[H_v/v : v\in K]$ where $|K|<\infty$ and each $H_v$ is a clique or independent set. We know that $K$ is a cograph. Assume that $K$ is the smallest. Let $T$ be the decomposition tree of $K$. By Lemma \ref{Kminimal}, $T$ is leaf minimal. For every $M\in T$ let $\mathcal{S}_{fps}(M)$ be the statement ``$P(M)$ is a finite poset substitution of chains or antichains". 

It is easy to verify that a poset whose comparability graph is a clique, resp an independent set, is a chain, resp an antichain. 
Therefore, the statement is true for the leaves of $T$. Let $M\in T$ be non-trivial and $M_1, \ldots, M_k$ the children of $M$. Also, assume that $\mathcal{S}_{fps}(M_i)$ holds for every $1\leq i\leq k$. We consider the following cases. 

\noindent\textbf{Case 1} $v(M)=0$. In this case $M$ is gs-disconnected. Let $x\in P(M)$. Then there is some $u\in M$ such that $x\in H_u$.  Since the $M_i$ are the children of $M$, $x\in P(M_i)$ for some $1\leq i\leq k$. Moreover, if $x\in P(M_i)$ and $y\in P(M_j)$, $i\neq j$, then $x$ and $y$ are disconnected because $v(M)=0$. Thus, $P(M)=\bigoplus_i P(M_i)$. That is $P(M)$ is the substitution of the $P(M_i)$ for a finite antichain. Since $\mathcal{S}_{fps}(M_i)$ holds for every $i$, $\mathcal{S}_{fps}(M)$ is true in this case.

\noindent\textbf{Case 2} $v(M)=1$.   
Since $M$ is edge minimal,  at most one of its children is a leaf $\{u\}$ such that $H_u$ is a clique, equivalently $P_u$ is a chain. If there is a leaf $u$ such that $P_u$ is a chain, we may suppose without loss of generality that $M_k=\{u\}$. Let $i, j\neq k$ be given with $i\neq j$. Then both $M_i$ and $M_j$ are gs-disconnected. By Lemma \ref{Disconnectedmodule}, we have $P(M_i) <_P P(M_j)$ or $P(M_j) <_P P(M_i)$.  Without loss of generality assume that $P(M_i) <_P P(M_j)$ if $i < j < k$. So, $P(\bigcup_{i<k}M_i)=+_{i<k}P(M_i)$. Again, by Lemma \ref{Disconnectedmodule}, for every $i < k$,  $P(M_i)$ is in a gap of $P_u$. Set $P_u^1:=\{x\in P_u : \{x\} <_P P(M_1)\}$, for every $1 < i < k$ set
$$P_u^i:=\{x\in P_u : P(M_{i-1}) <_P \{x\} <_P P(M_i)\}$$
and also set $P_u^k:=\{x\in P_u : P(M_{k-1}) <_P \{x\} \}$.  
Then, we have
$$P(M)=P_u^1 + P(M_1) + P_u^2 + P(M_2) + \cdots + P(M_{k-1}) + P_u^k,$$
where each $P_u^i$ is a chain. That is $P(M)$ is an alternate substitution of the $P(M_i)$ and the $P_u^i$ for a finite chain. 
Since $\mathcal{S}_{fps}(M_i)$ holds for every $i$ and since each $P_u^i$ is a chain, $\mathcal{S}_{fps}(M)$ is true in this case.  

Since the statement is true for every $M\in T$, $\mathcal{S}_{fps}(K)$ is true which means that $P$ is a finite poset substitution of chains or antichains. 
\end{proof}

We observed in Proposition \ref{Finiteps} that for every $M\in T$, $P(M)$ is a linear or direct sum of posets. Determining the structure of the siblings of $P(M)$ will enable us to obtain its sibling number. We show that a sibling of $P(M)$ has (up to isomorphism) the same context poset of $P(M)$ (a chain or antichain) replaced with siblings of poset blocks of $P(M)$. The following lemma shows that every embedding of $P(M)$ induces an injective mapping on the poset blocks of $P(M)$. Note that for each $u\in M$, $P_u$ is a chain or an antichain. When we say $P_u$ and $P_v$, $u, v\in M$, have the same {\em type}, it means that either both $P_u$ and $P_v$ are chains or both are antichains.

\begin{lemma} \label{Incomparableelement} 
Let $M\in T$ be edge minimal. The images of two distinct poset blocks of $P(M)$ under any embedding of $P(M)$ do not intersect the same poset block. 
\end{lemma}

\begin{proof}
Let $h$ be an embedding of $P(M)$. It suffices to show that for two distinct $u, v\in M$ and $x\in P_u, y\in P_v$, there exists some $z\in P(M)$ other than $x, y$ such that $z$ is comparable to one and only one of $x$ or $y$. Then $h(z)$ is comparable to one and only one of $h(x)$ or $h(y)$ meaning that $h(P_u)$ and $h(P_v)$ cannot intersect the same poset block because $P(M)$ is a poset substitution by Proposition \ref{Finiteps}.  

Note that a singleton as a poset block of $P(M)$ is both a chain and an antichain. Let two distinct $u, v\in M$ and $x\in P_u, y\in P_v$ be given. We consider the following cases:

\noindent\textbf{Case 1} $u, v$ are adjacent.  If both $P_u, P_v$ are chains, then the edge minimality of $M$ implies that $\{u, v\}$ is not a module of $M$. So, there is some $w\in M$ such that $w$ is adjacent to one of $u$ or $v$ and not to the other. Then picking some $z\in P_w$ proves the statement. If both $P_u, P_v$ are antichains, then one of them is non-trivial since otherwise they can be regarded as chains. Assume, without loss of generality, that $P_u$ is non-trivial. Then there is some $z\in P_u$ such that $z\perp x$ but $z$ and $y$ are comparable. Now suppose, without loss of generality, that $P_u$ is an antichain and $P_v$ is a chain. $P_u$ must be non-trivial since otherwise it can be regarded as a chain. Again, take some $z\in P_u$ with $z\perp x$. We know that $z$ and $y$ are comparable. 

\noindent\textbf{Case 2} $u, v$ are non-adjacent. If both $P_u, P_v$ are antichains, then the edge minimality of $M$ implies that $\{u,v\}$ is not a module of $M$. So, there is some $w\in M$ such that $w$ is adjacent to one of $u$ or $v$ but not to the other. Then pick some $z\in P_w$ which is a witness to the statement. Note that this also proves the case both $P_u, P_v$ are singletons. So, assume that both $P_u, P_v$ are chains and one of them, say without loss of generality $P_v$, is non-trivial. Take some $y\neq z\in P_v$ comparable to $y$. However, we have $z\perp x$. Finally, if $P_u$ and $P_v$ are of different types, then both of them are non-trivial. Therefore, there exists an element $z$ comparable to one and only one of $x$ or $y$. 
\end{proof}

By the lemma above one can define a graph isomorphism on $M$. 

\begin{lemma} \label{Distinctobjects}
Let $M\in T$ be edge minimal. Each embedding of $P(M)$ induces a graph isomorphism of $M$. Moreover, the embedding preserves the types of the poset blocks of $P(M)$. 
\end{lemma}  

\begin{proof} 
Let $h$ be an embedding of $P(M)$. First note that $M$ is finite. Therefore, there is a linear order $\leq_M$ on $M$. Define the mapping $\hat{h}:M\to M$ as follows: for each $u\in M$, select the least $v\in M$ w.r.t $\leq_M$ such that $h(P_u)\cap P_v\neq \emptyset$ and define $\hat{h}(u)=v$. If $h(P_u)\cap P_v\neq\emptyset$ and $h(P_u)\cap P_w\neq\emptyset$, then $v \leq_M w$ and $w \leq_M v$ which implies $v=w$. That is $\hat{h}$ is well-defined. For $u\neq v\in M$, by Lemma \ref{Incomparableelement}, $h(P_u)$ and $h(P_v)$ do not intersect the same poset block. Therefore, the mapping $\hat{h}$ is injective and since $M$ is finite $\hat{h}$ is onto. So, $\hat{h}$ is a bijection on $M$. Further, $u, v\in M$ are adjacent if and only if $\hat{h}(u)$ and $\hat{h}(v)$ are adjacent because $h$ preserves the order $\leq_P$. Thus, $\hat{h}$ is a graph isomorphism of $M$. 

Since $\hat{h}$ is well-defined, it is not the case that for some $u\in M$, $h(P_u)\cap P_v\neq \emptyset$ and $h(P_u)\cap P_w\neq\emptyset$ where $v\neq w$. It follows that $h(P_u)\subseteq P_{\hat{h}(u)}$ for every $u\in M$. The inclusion also implies that each poset block of $P(M)$ of any type is mapped into a poset block of $P(M)$ of the same type because a non-trivial chain cannot embed in an antichain and a non-trivial antichain cannot embed in a chain. 
\end{proof}

One more fact resulting from the lemma above is that any embedding of $P(M)$, $M\in T$, sends a singleton poset block of $P(M)$ to a singleton poset block because otherwise we get infinitely many chains or antichains as poset blocks of $P(M)$.

\begin{lemma} \label{Bijectiononconnected}
Let $M\in T$ be edge minimal such that $v(M)=0$. Every embedding of $P(M)$ permutes the gs-connected children of $M$ and fixes the possible gs-disconnected child of $M$.  Hence, a sibling of $P(M)$ is a direct sum of the siblings of the $P(M_i)$ where $M_i$ is a child of $M$.  
\end{lemma}

\begin{proof}
Let $h$ be an embedding of $P(M)$. 
Let $X:=\{M_1, \ldots, M_n\}$ be the set of gs-connected children of $M$ and $\{u\}$ the possible gs-disconnected child of $M$. Note that if the gs-disconnected child $\{u\}$ of $M$ exists, then $P_u$ is non-trivial and also all $P(M_i)$ are non-trivial because otherwise the edge minimality of $M$ fails. Let $M_i\in X$ be given and suppose that for some $x\in P(M_i)$, $h(x)\in P(M_j)$. For every $y\in P(M_i)$, $y$ is connected to $x$, so $h(y)$ is connected to $h(x)$ because $h$ preserves the path connecting $x$ and $y$. Therefore,  $h(P(M_i))\subseteq P(M_j)$ and since $P(M_i)$ is non-trivial, $h(P(M_i))\cap P_u=\emptyset$. Since $X$ is finite, there is a linear order $\leq_X$ on $X$. For every $M_i\in X$, select the least $M_j\in X$ w.r.t $\leq_X$ such that $h(P(M_i))\cap P(M_j)\neq\emptyset$ and define $\hat{h}(M_i)=M_j$. 
Since $M$ is gs-disconnected, it is not the case that for some $i$ and some $j\neq k$, $h(P(M_i))\cap P(M_j)\neq \emptyset$ and $h(P(M_i))\cap P(M_k)\neq\emptyset$. So, $\hat{h}$ is well-defined. Since $T$ is leaf minimal, the children of $M$ are edge minimal. Hence, each $P(M_i)$ is a finite poset substitution of chains or antichains. Let $\lambda_i$ be the cardinal of $M_i$ for each $i$. We have $\lambda_i > 0$ for every $i$. Now, suppose for two $i\neq j$, $h(P(M_i)), h(P(M_j))\subseteq P(M_k)$. By Lemma \ref{Distinctobjects}, we have $\lambda_i + \lambda_j\leq \lambda_k$.  This means that $\lambda_k > \lambda_i, \lambda_j$. A contradiction is immediate when $k=i$ or $k=j$ because then $\lambda_k > \lambda_k$. Assume that $k\neq i, j$. Since $\lambda_k > \lambda_i, \lambda_j$, $P(M_k)$ cannot be embedded into  $P(M_i)$ or $P(M_j)$. If $P(M_k)$ embeds in $P(M_k)$, then by Lemma \ref{Distinctobjects}, $\lambda_i+\lambda_j+\lambda_k \leq \lambda_k$, a contradiction to $\lambda_i, \lambda_j\neq 0$.  It follows that $P(M_k)$ embeds into some $P(M_l)$ where $l\neq i, j, k$. Continuing this, we get infinitely many gs-connected children of $M$, a contradiction. Thus, $\hat{h}$ is injective and since $X$ is finite, $\hat{h}$ is a bijection on $X$. It also follows that $h(P_u)\subseteq P_u$. Define $\hat{h}(\{u\})=\{u\}$.

Since $X$ is finite, for every $i$, there is some integer $m_i > 0$ such that $\hat{h}^{m_i}(M_i)=M_i$. This means that for every $i$, $P(M_i)\approx P(\hat{h}.i(M_i))$ where $\hat{h}.i$ is the orbit of $M_i$ under $\hat{h}$. Hence, a sibling of $P(M)$ is of the form $\bigoplus_iQ(M_i)\oplus P'_u$ where $Q(M_i)\approx P(M_i)$ for every $i$ and $P'_u\cong P_u$ because $P_u$ has only one sibling.  
\end{proof} 

For case $v(M)=1$ where $M\in T$ is edge minimal, we provide a result similar to Lemma \ref{Bijectiononconnected}.  

\begin{lemma} \label{Bijectionondisconnected}
Let $M\in T$ be edge minimal such that $v(M)=1$ and $h$ an embedding of $P(M)$. 
\begin{enumerate}
    \item For every gs-disconnected child $M_i$ of $M$ we have $h(P(M_i))\subseteq P(M_i)$.
    \item For any chain $P_u^i$ in the representation $P_u^1 + P(M_1) + P_u^2 + P(M_2) + \cdots + P(M_{k-1}) + P_u^k$ of $P(M)$ where $\{u\}$ is the unique possible gs-connected child of $M$, $h(P_u^i)\subseteq P_u^i$. 
\end{enumerate}
Hence, a sibling of $P(M)$ is of the form 
$$Q_u^1 + Q(M_1) + Q_u^2 + Q(M_2) + \cdots + Q(M_{k-1}) + Q_u^k $$
where $Q(M_i)\approx P(M_i)$ and $Q_u^i\approx P_u^i$ for every $i$. 
\end{lemma}

\begin{proof}
By Proposition \ref{Finiteps}, $P(M)$ is of the form 
$$P_u^1 + P(M_1) + P_u^2 + P(M_2) + \cdots + P(M_{k-1}) + P_u^k\ $$
where each $M_i$ is a gs-disconnected child of $M$ and the $P_u^i$ are the intervals of the unique chain $P_u$ where $\{u\}$ is the possible gs-connected child of $M$.   

(1) Let $i$ and $x\in P(M_i)$ be given. Since $M_i$ is gs-disconnected, take some $y$ in a component of $P(M_i)$ other than the component to which $x$ belongs. So, $x\perp y$. We have $h(x)\perp h(y)$. This implies that $h(x)\notin P_u^j$ where $1\leq j \leq k$. Suppose that $h(x)\in P(M_j)$ where $i\neq j$. Without loss of generality assume that $P(M_i) <_P P(M_j)$. Since $h(x)\perp h(y)$, we have $h(y)\in P(M_j)$. So, for every $y\in P(M_i)$ in a component other than the component to which $x$ belongs, we have $h(y)\in P(M_j)$. Exchanging the role of $x$ and $y$, for every $x\in P(M_i)$ in a component other than the component to which $y$ belongs, we have $h(x)\in P(M_j)$. It means that $h(P(M_i))\subseteq P(M_j)$. Let $\lambda_i, \lambda_j$ be the number of elements of $M_i, M_j$, respectively. By Lemma \ref{Distinctobjects} we have $\lambda_i \leq \lambda_j$. Thus, it is not the case that $P(M_j)$ embeds in $P(M_j)$ by $h$ because then we get $\lambda_i+\lambda_j \leq \lambda_j$, a contradiction to $\lambda_i > 0$. So, $h(P(M_j))\subseteq P(M_l)$ where $j < l$. Continuing this, we get infinitely many gs-disconnected children of $M$, a contradiction.  

(2) Let $i$ and some $x\in P_u^i$ be given. Suppose $h(x)\in P_u^j$ or $h(x)\in P(M_j)$ where $j\neq i$. Assume that $i < j$. The case $j<i$ is similar. Pick some gs-disconnected $M_l$ where $i \leq l < j$. Since $M_l$ is gs-disconnected, $P(M_l)\neq\emptyset$. Pick some $y\in P(M_l)$. We have $x <_P y$. By (1), $h(y)\in P(M_l)$. It follows that $h(y) <_P h(x)$, equivalently, $y <_P x$, a contradiction. 
Now assume that $h(x)\in P(M_i)$ (the argument for $P(M_{i-1})$ is similar). Since $M_i$ is gs-disconnected, take some $y\in P(M_i)$ such that $y$ belongs to a component of $P(M_i)$ other than the component to which $h(x)$ belongs. We have $x <_P y$ implying that $h(x) <_P  h(y)$. Also, $h(x) <_P h^2(x)$. Moreover, by (1), $h_{\upharpoonright P(M_i)}$ is an embedding from $P(M_i)$ into $P(M_i)$. This means that $h(x)$ and $y$ are mapped into the same component of $P(M_i)$. This contradicts Lemma \ref{Bijectiononconnected} because $v(M_i)=0$. 
\end{proof}

Now we are ready to deduce Thomass\'{e}'s Conjecture for a countable $N$-free poset whose comparability graph has one sibling. Bear in mind that by Theorem \ref{Dichchain}, a countable chain has either 1, $\aleph_0$ or $2^{\aleph_0}$ siblings.

\begin{theorem} \label{GtoP} 
Let $P$ be a countable $N$-free poset whose comparability graph has one sibling. Then $Sib(P)=1$ or $\aleph_0$ or $2^{\aleph_0}$.  
\end{theorem}  

\begin{proof}
Set $G:=CG(P)$. By Theorem \ref{Lexico}, $G$ is of the form $K[H_v/v : v\in K]$ where $K$ is finite and each $H_v$ is a clique or an independent set. Indeed, $K$ is a cograph. Assume that $K$ is the smallest. Let $T$ be the decomposition tree of $K$. By Lemma \ref{Kminimal},  $T$ is leaf minimal. For each $M\in T$ let $\mathcal{S}_{sib}(M)$ be the statement: ``$Sib(P(M))=1$  or $\aleph_0$ or $2^{\aleph_0}$".   We prove that $\mathcal{S}_{sib}(K)$ holds.

For every leaf $\{v\}$ of $T$, $P_v$ is a countable chain or antichain. Therefore, $\mathcal{S}_{sib}(\{v\})$ holds since an antichain has only one sibling and the statement holds for a countable chain by Theorem \ref{Dichchain}. Now, take $M\in T$ such that $M$ is not a leaf of $T$ and suppose that for each child $N$ of $M$, $\mathcal{S}(N)$ holds. If $v(M)=0$, then by Lemma \ref{Bijectiononconnected},  each sibling of $P(M)$ is a finite direct sum of siblings of the $P(N)$ where $N$ is a child of $M$; and if $v(M)=1$, then by Lemma \ref{Bijectionondisconnected}, each sibling of $P(M)$ is an alternate substitution of chains and the siblings of the $P(N)$ for a finite chain where $N$ is a non-trivial child of $M$. By the induction hypothesis and the fact that the statement holds for a countable chain,  $Sib(P(M))$ is the maximum of the sibling numbers of the components or the summands of $P(M)$. Thus, $\mathcal{S}_{sib}(M)$ holds. Since the statement is true for every $M\in T$,  $\mathcal{S}_{sib}(K)$ holds.   
\end{proof}

\section{Open Cases} 

We just proved the alternate Thomass\'{e} conjecture for countable $N$-free posets. We saw how the siblings of the sum of a labelled chain with no least element were constructed by its decomposition tree. The same technique was used for countable linear sums with infinitely many non-trivial summands. Also, there were other cases in both direct sums and linear sums in which we determined the exact sibling number. Moreover, Section \ref{Finiteposetsub} asserts that Thomass\'e's conjecture holds for countable $N$-free posets whose comparability graph has one sibling. Posing all these restrictions, we may ask the following.

\begin{problem}
Let P be a countable direct sum of connected $N$-free posets or a countable linear sum of $N$-free posets with property CCGC such that the comparability graph of $P$ has infinitely many siblings. Is it true that $Sib(P)=1$ or $\aleph_0$ or $2^{\aleph_0}$?
\end{problem}

\vspace{0.5cm}
\noindent{\bf \large Acknowledgements}

I would like to thank my PhD supervisors Professor Robert Woodrow and Professor Claude Laflamme for suggesting this problem and for their help and advice.

%
%

\vspace{1cm}

\textsc{Department of Mathematics and Statistics, University of Calgary, Calgary, Alberta, Canada, T2N 1N4}

{\em Email address:} \texttt{davoud.abdikalow@ucalgary.ca}

\end{document}